\documentclass[12pt]{article}

\usepackage{amssymb}
\usepackage{amsmath,amsfonts,amsthm,amstext}
\usepackage[all]{xy}
\usepackage{color}

{\bf }{\it }
\newtheorem{theorem}{Theorem}[section]{\bf }{\it }
\newtheorem{lemma}[theorem]{Lemma}{\bf }{\it }
\newtheorem{corollary}[theorem]{Corollary}{\bf }{\it }
\newtheorem{remark}[theorem]{Remark}{\it }{\rm }
\newtheorem{prop}[theorem]{Proposition}{\bf }{\it }
\newtheorem{example}[theorem]{Example}{\it}{\rm }
\newtheorem{conjecture} {Conjecture}{\it}{\rm }
\newtheorem{question}[conjecture]{Question}{\it}{\rm }
\newtheorem{problem}[conjecture]{Problem}{\it}{\rm }

\newcommand{\Z}{\mathbb{Z}}
\newcommand{\F}{\mathbb{F}}

\newcommand{\N}{\mathbb{N}}
\newcommand{\CC}{\mathbb{C}}
\newcommand{\Q}{\mathbb{Q}}

\newcommand{\SL}{\rm {SL}}

\newcommand{\G}{\Gamma}
\DeclareMathOperator{\Hom}{{Hom}}

\DeclareMathOperator{\df}{def} \DeclareMathOperator{\im}{Im}

\DeclareMathOperator{\rk}{\mathrm{rk}}
\DeclareMathOperator{\cd}{\mathrm{cd}}\DeclareMathOperator{\vcd}{\mathrm{vcd}}

\newcommand{\ochi}{\overline{\chi}}

\numberwithin{equation}{section}

\begin{document}

\pagestyle{headings}
\title{Normal Subgroups of Profinite Groups of Non-negative
Deficiency
}

\author{\frame{Fritz Grunewald}, Andrei Jaikin-Zapirain,\\
Aline G.S. Pinto
~\&~ Pavel A. Zalesski
}

\maketitle

\begin{abstract}
 The principal focus of the paper is to show that the
existence of a finitely generated normal subgroup of infinite
index in a profinite group $G$ of non-negative deficiency gives
rather strong consequences for the structure of $G$. To make this
precise we introduce the notion of $p$-deficiency ($p$ a prime)
for a profinite group $G$. This concept is more useful in the
study of profinite groups then the notion of deficiency. We prove
that if the $p$-deficiency of $G$ is positive and $N$ is a
finitely generated normal subgroup such that the $p$-Sylow
subgroup of $G/N$ is  infinite and $p$ divides the order of $N$
then we have $\cd_p(G)=2$, $\cd_p(N)=1$ and $\vcd_p(G/N)=1$ for
the cohomological $p$-dimensions; moreover either the $p$-Sylow
subgroup of $G/N$ is virtually cyclic or the $p$-Sylow subgroup of
$N$ is cyclic. If $G$ is a profinite Poincar\'e duality group  of
dimension $3$ at a prime $p$ ($PD^3$-group) we show that for $N$
and $p$ as above
 either $N$ is $PD^1$ at $p$ and $G/N$ is virtually $PD^2$ at $p$
or $N$ is $PD^2$ at $p$ and $G/N$ is virtually $PD^1$ at $p$. In
particular if $G$ is pro-$p$ then either $N$ is infinite cyclic
and $G/N$ is virtually Demushkin  or $N$ is Demushkin and $G/N$ is
virtually  infinite cyclic. We apply this results to deduce
structural information on the profinite completions of ascending
HNN-extensions of free groups. We also give some implications
of our theory to the congruence kernels of certain arithmetic groups.
\\
2000. \emph{Mathematics Subject Classification}
Primary: 20E18.
\end{abstract}

\pagebreak

\tableofcontents

\section{Introduction}

If a connected compact manifold $M$ admits a fibration over a
compact base manifold $B$ with a compact manifold $F$ as a fiber
its fundamental group $\G=\pi_1(M)$ satisfies an exact sequence
\begin{equation}\label{eqa1}
\langle 1\rangle \to N\to \G=\pi_1(M) \to \G/N\to \langle 1\rangle
\end{equation}
where $N$, being an image of $\pi_1(F)$, is a normal subgroup
which is finitely generated as a group. Moreover there are many
interesting situations where the quotient $\G/N$ is the infinite
cyclic group $\Z$. This for example happens when $M$ fibers over
the circle.
 In fact J. Hempel and W. Jaco have proved

\begin{theorem}\label{theohemp} ({\bf Hempel, Jaco \cite{HW}})
Let $\G$ be the fundamental group of  compact 3-manifold $M^3$
(possibly with boundary)
and $N$ a finitely generated normal subgroup of $\G$ of infinite index.
Then $N$ is isomorphic to the fundamental group of a compact, possibly bounded
$2$-manifold. In case $N$ is not infinite cyclic then the quotient $\G/N$
has an infinite cyclic group of finite index.
\end{theorem}

Furthermore, Hempel and Jaco add some results on the geometric
structure of the manifold $M^3$. The group theoretic result in
Theorem \ref{theohemp} obtained a group theoretic proof by J.
Hillman in \cite{Hilman} which also provided an important
generalization. In fact J.  Hillman proved Theorem \ref{theohemp}
for Poincar\'e duality groups $\G$ of dimension $3$. This
generalizes the case of a compact $3$-manifold $M^3$ without
boundary. Hillman \cite{Hilman} was also able to give the cases
where $M^3$ is a compact $3$-manifold with boundary an appropriate
generalization.

In this paper we
\begin{itemize}
\item prove results analogous to Theorem \ref{theohemp} for
profinite groups, \item study applications of our profinite
results to discrete groups.
\end{itemize}

We for example establish in Section \ref{PDdrei} the following
profinite version of the result of Hillman:

\begin{theorem}\label{thpd3}
Let $G$ be a profinite $PD^3$-group at a prime $p$  and $N$ be a
finitely generated normal  subgroup of $G$ such that the $p$-Sylow $(G/N)_p$ is
infinite and $p$ divides $|N|$. Then either $N$ is $PD^1$ at $p$
and $G/N$ is virtually $PD^2$ at $p$ or $N$ is $PD^2$ at $p$ and
$G/N$ is virtually  $PD^1$ at $p$.
\end{theorem}

The pro-$p$ version of this theorem   reads as follows:

\begin{corollary} Let $G$ be a pro-$p$  $PD^3$-group and $N$ be a
non-trivial finitely generated normal  subgroup of $G$ of infinite index. Then
either $N$ is infinite cyclic  and $G/N$ is virtually Demushkin
or $N$ is Demushkin and $G/N$ is virtually  infinite cyclic.
\end{corollary}

To explain the content of this paper in more detail we have to introduce
some concepts from group theory.
The {\it deficiency}  $\df(G)$ of a group $G$
is the largest integer $k$ such that there exist a (finite)
presentation of $G$ with the number
of generators minus the number of relations equal to $k$.
The groups of non-negative deficiency form an
important class of finitely presented groups.
It contains many important families of
examples coming from geometry: fundamental groups
of compact 3-manifolds, knot groups, arithmetic groups of rank 1, etc.

In this paper we  study   profinite groups of
non-negative deficiency. The deficiency for profinite groups is
defined in the same way as for discrete groups.
We consider presentations
in the category of profinite groups, i.e., saying
generators we mean topological generators. We note that any
presentation of a group is a presentation of its profinite
completion and so the profinite completion of a group of
non-negative deficiency  is a profinite group of non-negative
deficiency. Therefore, the profinite completions of groups
mentioned above are in the range of our study.  Moreover, as was shown by
Lubotzky \cite{Lu}, Corollary 1.2, all projective groups (i.e.
profinite groups of cohomological dimension 1)   have non-negative
deficiency as well.

The principal focus of our study here is to show that the existence of
a finitely generated normal subgroups of infinite index in a
profinite group $G$ of non-negative deficiency gives
rather strong consequences for the structure of $G$.
We are ready to state the principal results for groups of positive deficiency.

\begin{theorem}\label{tmaintro}
Let $G$ be a finitely generated profinite group with positive deficiency
and $N$ a finitely generated normal subgroup such that the $p$-Sylow
subgroup $(G/N)_p$ is  infinite and $p$ divides the order of
$N$. Then   either the $p$-Sylow subgroup of $G/N$ is
virtually cyclic or the $p$-Sylow subgroup of $N$ is cyclic.
Moreover, $\cd_p(G)=2$, $\cd_p(N)=1$ and $\vcd_p(G/N)=1$, where
$\cd_p$ and $\vcd_p$ stand for cohomological $p$-dimension and virtual
 cohomological $p$-dimension respectively.
\end{theorem}

To prove this theorem we introduce in Section \ref{defi}
the concept of $p$-deficiency
$\df_p(G)$ for a prime $p$ and a profinite group $G$. These new invariants
are more suitable to the study of profinite groups then just the deficiency.
The example of pro-$p$ groups (which all have non-positive deficiency as
profinite groups) already makes it clear that our result requires a more
subtle approach.
Section \ref{defi} contains also results interconnecting the deficiency
of a profinite group with its various $p$-deficiencies.

Theorem \ref{tmaintro} has the following immediate consequence.

\begin{corollary}\label{profiniteintro}
Let $G$ be a finitely generated profinite group of positive
deficiency and $N$ a finitely generated normal subgroup of $G$
such that the $p$-Sylow subgroup $(G/N)_p$ is infinite whenever
the prime $p$ divides $|N|$. Then $N$ is projective.
\end{corollary}

The  pro-$p$ version of of Theorem \ref{tmaintro} is proved in the paper \cite{HS} of Hillmann
and Schmidt and its abstract version for a discrete group $\Gamma$ is known only under
further restrictions either on $N$ or on $\G/N$ (see \cite{Bieri}).

A group $\G$ is called knot-like if $\G/[\G,\G]$ is infinite
cyclic and the deficiency satisfies ${\rm def}(\G)=1$. These two
properties are possessed by any knot group, i.e., the fundamental
group of the complement  of a knot in the $3$-sphere $S^3$. It was
conjectured by E. Rapaport-Strasser in \cite{rapap} that if the
commutator group $\G'=[\G,\G]$ of a knot-like group $\G$ is
finitely generated then $\G'$ should be free. This conjecture is
true as it was proved by D.H. Kochloukova in \cite{K2}.

The next corollary shows that the profinite version of the
Rapaport-Strasser conjecture is also true. In fact our result is
stronger then just the profinite version of the conjecture since
we  do not assume $G/[G,G]$ to be cyclic and assume positive deficiency
rather than for deficiency one.

\begin{corollary}\label{coraintro}  Let $G$ be a finitely generated profinite
group of positive deficiency whose commutator subgroup $[G,G]$ is
 finitely generated. Then $\df(G)=1$ and $[G,G]$ is projective. Moreover,
 $\cd(G)=2$
unless $G=\widehat{\mathbb{Z}}$. \end{corollary}


The next class of groups where results apply are ascending
HNN-extensions of free groups (also known as mapping tori of free
group endomorphisms). Groups of this type often appear in group
theory and topology and were extensively studied (see \cite{FH},
\cite{BS} for example). In particular, many one-relator groups are
ascending HNN-extensions of free groups and many of such  groups
are hyperbolic. Corollary \ref{profiniteintro} allows to establish
the structure of the profinite completion of this important class
of groups.

\begin{theorem} Let $F=F(x_1,\ldots x_n)$ be a free group of finite rank
$n$ and $f:F\longrightarrow F$ an endomorphism. Let
 $\G=\langle F,t\mid x_i^t=f(x_i)\rangle$ be the HNN-extension.
Then the
profinite completion of $\G$ is $\widehat \G=P\rtimes
\widehat{\mathbb{Z}}$, where $P$ is projective.
$P$ is free profinite of rank $n$ if and only if $f$ is an automorphism.
\end{theorem}

As a corollary we obtain
the following surprising consequence.

\begin{theorem} An ascending HNN-extension $\G$ of a
free group is good.\end{theorem}

 The group $\G$ is called {\it $p$-good} if the
homomorphism of cohomology groups
$$H^n(\widehat \G,M)\rightarrow H^n( \G,M)$$
induced by the natural homomorphism $\G\rightarrow \widehat \G$ of
$\G$ to its profinite completion $\widehat \G$ is an isomorphism
for every finite $p$-primary $\G$-module $M$.  The group $\G$ is
called {\it good} if is $p$-good for every prime $p$.  This
important concept was introduced by J-P. Serre in \cite[Section
I.2.6]{Serre}. In his book Serre explains the fundamental role
that goodness plays in the comparison of properties of a group and
its profinite completion. In Section 5 we prove that all
arithmetic Kleinian groups are good.

\begin{theorem}  Let $\Gamma$ be an arithmetic Kleinian group in $\SL_2(\CC)$. Then $\Gamma$ is  good.
\end{theorem}

\medskip

\bigskip
{\it Acknowledgements:} We thank Wilhelm Singhof for conversations
on the subject.

The second author  was partially supported by the Spanish Ministry of
  Education, grant MTM2008-06680.
The third and forth authors were partially supported by ``bolsa de
produtividade de pesquisa" from CNPq, Brazil.

\section{Preliminaries}

This section contains certain preliminary lemmas which will be of use later.
Also we fix the following standard notations for this paper.

\begin{itemize}
 \item $\F_p$ - stands for the field of $p$ elements;
 \item $\Q_p$ - is the field of $p$-adic numbers;
 \item $\Z_p$ - is the ring of $p$-adic integers;
 \item $\widehat \G$ - is the the profinite completion of a group $\G$;
 \item $\G_{\hat p}$ - is the pro-$p$ completion of a group $\G$;
 \item $G_p$ - is the $p$-Sylow subgroup of a profinite group $G$;
 \item $G_{[p]}$ - is the maximal pro-$p$ quotient of a profinite group $G$;
 \item $\Z_p[[G]]$ - is the completed group ring, i.e.,
$\Z_p[[G]]= \lim\limits_{\displaystyle\longleftarrow}{}_{} \Z_p[G_j]$
which is the inverse
limit of ordinary group rings with $G_j$ ranging over all the finite
quotients  of $G$;
 \item $[G:H]_p$ - is the largest power of $p$ dividing the index
 $[G:H]$ of $H$ in $G$;
 \item  $\dim M$ - is the length of $M$ as $\Z$-module.
\item $d(G)$ - the minimal number of topological generators of a
profinite group $G$.

\end{itemize}

\subsection{Homology and cohomology of profinite groups}

In this section we collect some notation and well known facts
concerning the homology and cohomology of profinite groups. If we
do not say the contrary module means left module.

Let $G$ be a profinite group and $B$ a profinite
$\Z_p[[G]]$-module. The $i$th homology group $H_i(G,B)$ of $G$
with coefficients in $B$ is defined by
$$
H_i(G,B)=\mathrm{Tor}_i^{\Z_p[[G]]}(\Z_p,B),
$$
where $\mathrm{Tor}_i^{\Z_p[[G]]}(\Z_p, - )$ is the $i$th derived
functor of the right exact covariant functor of
$\Z_p\widehat{\otimes}_{\Z_p[[G]]}-$ from the category of left
profinite $\Z_p[[G]]$-modules to profinite $\Z_p$-modules.

Similarly, given a discrete $\Z_p[[G]]$-module $A$, the $i$th
cohomology group $H^i(G,A)$ of $G$ with coefficients in $A$ is
defined by
$$
H^i(G,A)=\mathrm{Ext}^i_{\Z_p[[G]]}(\Z_p,A),
$$
where $\mathrm{Ext}^i_{\Z_p[[G]]}(\Z_p, - )$  is the $i$th derived
functor of the left exact covariant functor
$\mathrm{Hom}_{\Z_p[[G]]}(\Z_p, - )$ from the category of left
discrete $\Z_p[[G]]$-modules to discrete $\Z_p$-modules. It can be
calculated by using either projective resolutions of the trivial
module $\Z_p$ in the category of profinite $\Z_p[[G]]$-modules or
injective resolutions of the discrete $\Z_p[[G]]$-module $A$.

The categories of profinite and torsion discrete
$\Z_p[[G]]$-modules are dual via the Pontryagin duality
(\cite[5.1]{RZ}) and so are $H_i(G,-)$ and $H^i(G,-^*)$, where
$^*$ stands for $\mathrm{Hom}(-,\mathbb{Q}/\Z)$. Hence we have
that $H_i(G,\mathbb{F}_p)$ and $H^i(G,\mathbb{F}_p)$ are vector
spaces over $\mathbb{F}_p$ of the same dimension.

By definition, a profinite group $H$ is of type $p$-$FP_{m}$ if
the trivial profinite $\Z_p[[H]]$-module $\Z_p$ has a profinite
projective resolution over $\Z_p[[H]]$ with all projective modules
in dimensions $\leq m$ finitely generated.  We say that $H$
is of type $p$-$FP_{\infty}$ if $H$ is of type $p$-$FP_m$ for
every $m$. If $G$ is a profinite group of type $p$-$FP_m$ and $M$
is a $G$-module of $p$-power order, then the cohomology groups
$H^i(G,A)$ (and therefore $H_i(G,A)$) are finite for all $i\leq
m$. If $G$ is a pro-$p$ group, the fact that $\Z_p[[G]]$ is a
local ring implies that $G$ is of type $FP_m$ if and only if
$H^i(G, \F_p)$ are finite for $i \leq m$ (see Theorem \ref{critfp} for an analogous result for an arbitrary profinite group).


The cohomological $p$-dimension of a profinite group $G$ is the
lower bound of the integers $n$ such that for every discrete
torsion $G$-module $A$, and for every $i>n$, the $p$-primary
component of $H^i(G,A)$ is null. We shall use the standard
notation ${\rm cd}_p(G)$ for cohomological $p$-dimension of the
profinite group $G$. The cohomological dimension ${\rm cd}(G)$ of
$G$ is defined as the supremum ${\rm cd}(G)={\rm sup}_p({\rm
cd}_p(G))$ where $p$ varies over all primes $p$.

The next proposition gives a well-known characterization  for
${\rm cd}_p$ (see {\cite[Prop. I.$\S$3.1.11 and
I.$\S$4.1.21]{Serre}} and  \cite[Proposition 7.1.4]{RZ}).

\begin{prop}\label{projdim} Let $G$ be a profinite group, $p$ a prime and $n$
an integer. The following properties are equivalent:
\begin{itemize}
 \item[{\rm 1.}] ${\rm cd}_p(G)\leq n$;
 \item[{\rm 2.}] $H^i(G,A)=0$ for all $i>n$ and every discrete $G$-module $A$ which
 is a $p$-primary torsion module;
 \item[{\rm 3.}] $H^{n+1}(G,A)=0$ when $A$ is simple discrete $G$-module annihilated by
 $p$;
 \item[{\rm 4.}] $H^{n+1}(H,\F_p)=0$ for any open subgroup $H$ of
 $G$.
 \item[{\rm 5.}] The projective dimension of the trivial $\F_p[[G]]$-module $\F_p$ is $\le n+1$.
\end{itemize}
\end{prop}

Note that if $G$ is pro-$p$ then there is only one simple discrete
$G$-module annihilated by $p$, namely the trivial module
$\mathbb{F}_p$.

The Lyndon-Hochschild-Serre spectral sequence will be the  most
important tool in this paper. We will give its brief description
and the most important consequences. For the details see
\cite{RZ}.

\begin{theorem}[{\cite[Thm. 7.2.4]{RZ}}] Let $N$ be a normal
closed subgroup of a profinite group $G$, and let $A$ be a
discrete $G$-module. Then there exists a spectral sequence
$\mathbb{E}=(E_t^{r,s})$ such that $$E^{r,s}_2\cong
H^r(G/N,H^s(N,A))\Rightarrow H^n(G,A).$$
\end{theorem}

\begin{corollary}\label{consLHS}
Let $N$ be a normal closed subgroup of a profinite group $G$, and
let $A$ be a discrete $G$-module. Then the following holds.
\begin{enumerate}
 \item[{\rm (1)}] There exists always
 a five term exact sequence
 $$
 \begin{array}{l}
 0\to H^1(G/N,A^N)\to H^1(G,A)\to \medskip \\
 H^1(N,A)^{G/N} \to H^2(G/N,A^N) \to H^2(G,A).
 \end{array}
 $$
 \item[{\rm (2)}] If $H^i(N,A)=0$ for all $i\ge
 1$, then $H^i(G,A)\cong H^i(G/N,A^N)$.
 \item[{\rm (3)}] If $H^1(N,A)=0$, then $H^1(G/N,A)\cong H^1(G,A)$ and $H^2(G/N,A)\hookrightarrow
 H^2(G,A)$.
  \item[{\rm (4)}] If $\cd_p(G/N)\le 1$ and $A$ is a $p$-primary torsion, then there exists the following exact sequence
 $$
0\to  H^1(G/N, H^i(N,A))\to H^{i+1}(G,A)\to  H^{i+1}(N,A)^{G/N}\to 0.
 $$
 \item[{\rm (5)}] If $G/N^p[N,N]$ splits as the direct product $G/N\times
 N/N^p[N,N]$ then
 $$
 H^2(G/N,\F_p)\oplus
 H^1(G/N,H^1(N,\F_p))\hookrightarrow H^2(G,\F_p).
 $$
\end{enumerate}
\end{corollary}

\begin{proof}
Items (1) and (2) can be found in \cite[Corollary 7.2.5]{RZ}; (3)
follows directly from (1).  We sketch the proofs of items (4) and (5).

(4)  Consider the  Lyndon-Hochschild-Serre spectral sequences
$(E^{\bullet,\bullet}_\bullet, d_\bullet)$ for the $G$-module $A$
associated to the extension $1\to N \to G\to G/N\to 1$. From the
hypothesis it follows that $E_2^{t,\bullet}=0$ for $t\ge 2$. Hence
the spectral sequence $( E^{\bullet,\bullet}_\bullet, d_\bullet)$
collapses at the $ E_2$-term.

(5) Consider the  Lyndon-Hochschild-Serre spectral sequences
$(E^{\bullet,\bullet}_\bullet, d_\bullet)$ for the trivial
$G$-module $\F_p$ associated to the extension $1\to N \to G\to
G/N\to 1$, and $(\bar E^{\bullet,\bullet}_\bullet, d_\bullet)$
for the trivial $G/N^p[N,N]$-module $\F_p$ associated to the
extension $1\to N/N^p[N,N] \to G/N^p[N,N]\to G/N\to 1$. Note that
by construction, the second extension is a direct product, and
thus the spectral sequence $(\bar E^{\bullet,\bullet}_\bullet,
d_\bullet)$ collapses at the $\bar E_2$-term, i.e., $\bar d_t = 0$
 for all $t \ge  2$ (\cite[p. 96, Exercise 7]{NSW}). Denote by
 $\pi_2^{n,k}$ the natural map $\bar E_2^{n,k}\to E_2^{n,k}$. Then
 for each pair $n,k$
 we have
 the following commutative diagram
 $$\begin{array}{ccc}
 \bar E_2^{n,k} & \stackrel{0}{\longrightarrow}& \bar E_2^{n+2,k-1}\medskip\\
 \downarrow ^{\pi_2^{n,k}} &&\downarrow ^{\pi_2^{n+2,k-1}}\\
 E_2^{n,k}&\stackrel{d_2^{n,k}}{\longrightarrow}& E_2^{n+2,k-1}\end{array}
 $$
Since $\pi_2^{0,1}$  is an isomorphism, it follows from the
diagram that $d_2^{0,1}=0$ and so $E_3^{2,0}= E_2^{2,0}\cong
H^2(G/N,\F_p)$. In the same way, as $\pi_2^{1,1}$ is an
isomorphism, we get that $d_2^{1,1}=0$ and so
$E_3^{1,1}=E_2^{1,1}\cong H^1(G/N,H^1(N,\F_p))$.

\end{proof}
We will use the following result that relates the cohomological
$p$-dimensions of $G$, $N$ and $G/N$.
\begin{theorem}[{\cite[Thm. 1.1]{WZ}}]\label{comptes}
Let $G$ be a profinite group of finite cohomological $p$-dimension
$cd_p(G)=n$ and let $N$ be a closed normal subgroup of $G$ of
cohomological $p$-dimension $cd_p(N)=k$ such that
$H^k(N,\mathbb{F}_p)$ is nonzero and finite. Then $G/N$ is of
virtual cohomological $p$-dimension $n-k$.
\end{theorem}
In the case when $N$ is of cohomological $p$-dimension 0 or 1 we
have the following corollary.
\begin{corollary} Let $G$ be a profinite group of finite cohomological $p$-dimension
$cd_p(G)$ and let $N$ be a finitely generated closed normal
subgroup of $G$ of cohomological $p$-dimension $cd_p(N)\leq 1$.
Then $G/N$ is of virtual cohomological $p$-dimension
$cd_p(G)-cd_p(N)$. \label{corwz}
\end{corollary}

\begin{proof} If $p$ does
not divide $|N|$,  then Corollary \ref{consLHS}(2) implies the
isomorphisms $H^{k}(U/N,\mathbb{F}_p)\cong H^{k}(U,\mathbb{F}_p)$
for all open subgroups $U$ of $G$ which contain $N$ and all $k$.
Now from Proposition \ref{projdim} it follows that
$cd_p(G/N)=cd_p(G)$.

If $p$ divides $|N|$ then, by Proposition \ref{projdim},
$cd_p(N)=1$ and so there exists an open subgroup $V$ of $N$ such
that $H^1(V,\mathbb{F}_p)\neq 0$. We can find an open subgroup $U$ of $G$
such that $U\cap N= V$ and  apply Theorem \ref{comptes} to $U$ and
$V$.  \end{proof}

\subsection{The deficiency}\label{defi}

If $G$ is a finitely generated group (profinite group), then we
say that $G$ is of {deficiency} $k$ if there exist a presentation
$\pi:F\to G$ of $G$ of deficiency $k$, i.e. such that the rank of
the free group (profinite group) $F$ minus the number of
generators of $\ker\pi$ as a (closed) normal subgroup of $F$ is
$k$. If $r>0$, it is possible to add further relations that are
consequences of the original ones, so a presentation of smaller
deficiency is obtained. Thus by our definition a group of
deficiency $k+1$ is also of deficiency $k$. We denote by $\df(G)$
the greatest $k$ such that $G$ is of deficiency $k$. Note that for
any group $G$ one has $\df(G)\le \df(\widehat G)$.

Below we shall introduce other invariants  that help to describe
more precisely properties of profinite groups; some of them are
taken from \cite{Lu}. The necessity of this already can be seen
from the fact that any pro-$p$ group (including free pro-$p$
groups) as a profinite group is of non-positive deficiency.

Let $G$ be a finitely generated profinite group. Denote by $d(G)$
its minimal number of generators. If $M$ is a non zero finite
$G$-module we denote by $\dim M$ the length of $M$ as $\Z$-module
and put
$$
\ochi_2(G,M)=\frac { - \dim  H^2(G,M) + \dim H^1(G,M) - \dim
H^0(G,M) } { \dim(M) }.
$$
Also we introduce a cohomological variation of the minimal number
of generators. We put
$$
\ochi_1(G,M)=\frac {\dim H^1(G,M)-\dim H^0(G,M) } { \dim(M) }.
$$
If $M=\{0\}$, then we agree that $\ochi_2(G,M)=+\infty $ and
$\ochi_1(G,M)= -\infty$. In fact,
$\ochi_k(G,M)=-\chi_k(G,M)/\dim(M)$, where $\chi_k(G,M)$ is the
partial Euler characteristic of the $G$-module $M$.

\begin{example}\label{xi(F)}
\em{ a) If $G$ is a finitely generated profinite group and $d(G)$
denotes the minimal number of generators of $G$, we have
 $$
 \dim H^0(G,M)=\dim M^G \textrm{\  and\ }
 \ochi_1(G,M)\leq
 d(G)-1;
 $$

 \medskip \noindent b) If $F$ be a free profinite group of rank $\rk(F)$,
 $$
 \ochi_2(F,M)=\ochi_1(F,M)=\rk(F)-1.
 $$

Indeed, for any finitely generated profinite group $G$, the
augmentation ideal of $\Z_p[[G]]$ can be generated as
$\Z_p[[G]]$-module by $d(G)$ elements. So we write a partial free
$\Z_p[[G]]$-resolution
$$
\mathcal{F}: ~~ F_2 \to \Z_p[[G]]^{d(G)}\stackrel{\delta}{\to}
\Z_p[[G]] \to \Z_p \to 0
$$
and apply $Hom_{\Z_p[[G]]}(-,M)$ to the complex
$\mathcal{F}_{del}$ obtained by suppressing $\Z_p$. The cohomology
of $G$ (in small dimensions) is the cohomology of the complex
$$
Hom_{\Z_p[[G]]}(\mathcal{F}_{del},M): ~~~~ 0 \to M
\stackrel{\phi}\to M^{d(G)}\stackrel{\psi}\to
Hom_{\Z_p[[G]]}(F_2,M).
$$
Then
$$\dim H^0(G,M)=\dim \ker(\phi) =\dim M^G$$
and
$$
\begin{array}{rl}
\dim(M) \ochi_1(G,M)&=\dim H^1(G,M)-\dim H^0(G,M) \medskip \\
& = \dim \ker(\psi)- \dim ~\mathrm{Im}(\phi)-\dim ~\ker(\phi)
\medskip \\
& = \dim \ker(\psi) - \dim M\medskip \\
& \leq d(G) \dim(M)- \dim(M) = \dim(M)(d(G)-1).
\end{array}
$$
This proves item a). Now, if $G=F$ is free we have $d(G)=\rk(F)$
and the augmentation ideal of $\Z_p[[F]]$ is free as
$\Z_p[[F]]$-module of rank $\rk(F)$. So we can take $F_2=0$ in
$\mathcal{F}$ and the above inequality is an equality. The first
equality in b) we obtain from $H^2(F,M)=0$. }
\end{example}

\begin{lemma}\label{short}
Let $G$ be a finitely generated profinite group. Let $0\to M'\to
M\to M''\to 0$ be an exact sequence of finite $G$-modules. Then
$$
\ochi_2(G,M)\ge \min\{ \ochi_2(G,M'),\ochi_2(G,M'')\}
$$
and
$$
\ochi_1(G,M)\le \max\{ \ochi_1(G,M'),\ochi_1(G,M'')\}.
$$
\end{lemma}

\begin{proof}
We will prove the first inequality. The second one
is proved using the same method. The short exact sequence
$$
0\to  M'\to M\to M''\to 0
$$
gives a long exact sequence in cohomology
$$
0\to H^0(G,M')\to H^0(G, M)\to \ldots\to H^2(G,M)\to H^2(G,
M'')\stackrel{\delta_3}{\longrightarrow} \ldots
$$
Counting dimension we have
 $$
 \begin{array}{ll}
 \dim H^0(G,M) = & \dim H^0(G,M') + \dim H^0(G,M'') - \dim H^1(G,M')
 \medskip \\ & + \dim H^1(G,M) - \dim H^1(G,M'') + \dim H^2(G,M') \medskip \\ &
 - \dim H^2(G,M) + \dim H^2(G,M'') - \dim \im \delta_3
 \end{array}
 $$
from where we obtain
$$
\dim(M')\ochi_2(G,M')-\dim(M) \ochi_2(G,M)+
\dim(M'')\ochi_2(G,M'')\le 0.
$$
Hence,
$$
\begin{array}{c}
\ochi_2(G,M)  \ge \frac{\dim(M')\ochi_2(G,M') + \dim(M'')
\ochi_2(G,M'')}{\dim(M)} \medskip  \ge \min\{ \ochi_2(G,M'),
\ochi_2(G,M'')\}.
\end{array}
$$
\end{proof}

Let $G$ be a finitely generated profinite group and let
$$1\to R\to F\to G\to 1$$ be a presentation of $G$, where $F$ is a
$d(G)$-generated free profinite group. Then $\bar R :=
R/[R,R]\cong H_1(R, \hat{\mathbb{Z}})$ as $G$-modules and it is
called the relation module of $G$. By Shapiro's lemma $\bar R\cong
H_1(F,\hat \Z[[G]])$ and, as it is shown in \cite{Lu}, it does not
depend on the presentation.

If we decompose $\bar R$ as the product of its $p$-primary
components $\bar R= \prod_p \bar R_p\,,$ one gets $\bar R_p\cong
H_1(F,\Z_p[[G]])$. Since the augmentation ideal of $\Z_p[[F]]$ is
a free profinite $\Z_p[[F]]$-module of rank $d(G)$, we have the
exact sequence
$$
0\to \Z_p[[F]]^{d(G)}\to \Z_p[[F]] \to \Z_p \to 0
$$
 of free $\Z_p[[F]]$-modules. Applying the functor
$\Z_p[[F/R]]\widehat\otimes_{\Z_p[[F]]}-$, we obtain the following
exact sequence
\begin{equation} \label{relmod}
0\to \bar R_p\to \Z_p[[G]]^{d(G)}\stackrel{\varphi}{\to}
\Z_p[[G]]\to \Z_p\to 0,
\end{equation}
since $\ker(\varphi)=H_1(F,\Z_p[[G]])$. If $M$ is a $ \Z_p[[
G]]$-module, we denote by $d_G(M)$ its minimal number of
generators. Thus, the last exact sequence may be rewritten as
follows.
\begin{equation}\label{presentation} \Z_p[[G]]^{d_G(\bar R_p)}\to \Z_p[[G]]^{d(G)}\to \Z_p[[G]]\to \Z_p\to
0.\end{equation}

The following result is proved in \cite{Lu}.
\begin{theorem}\label{profinitedeficiency}
Let $G$ be a finitely generated profinite group. Then, for a fixed
prime $p$,
$$
d_G(\bar R_p)= \min_M\{d(G)-1-\ochi_2(G,M)\}
$$
 where $M$ runs over all irreducible $\Z_p[[G]]$-modules, and
 $$
 d_G(\bar R) = \max_p d_G(\bar R_p).
 $$
Moreover, $\df (G)= d(G)-d_G(\bar R)$ unless $\bar R=0$   and $G$
is not free, in which case $\df(G)=d(G)-1$.
\end{theorem}
Combining this theorem with Lemma \ref{short}, we obtain that if
$\bar R\ne 0$, then
\begin{equation}\label{def(G)}
\df (G)=\min\{1+\ochi_2(G,M)|M \textrm{\ is a finite
$G$-module}\}.
\end{equation}
If $N$ is a closed subgroup of $G$ and $\mathcal{M}_p(N)$ is the
set of all finite $\Z_p[[G]]$-modules on which $N$ acts trivially,
we introduce the following invariants:
$$
\df_{p}(G,N)=\min_{M \in \mathcal{M}_p(N)}\{1+\ochi_2(G,M)\}
$$
 and
$$
d_{p}(G,N)=\max_{M\in\mathcal{M}_p(N)}\{1+\ochi_1(G,M)\}.
$$
For simplicity, we put
$$
\df_p(G)=\df_{p}(G,\{1\}) ~~\mbox{ and }~~ d_p(G)=d_{p}(G,\{1\}).
$$
The number  $\df_p(G)$ is called the {\bf $p$-deficency} of $G$.  Comparing this invariant with the deficiency of $G$ we observe
that
\begin{equation}\label{p_def}
\df(G) \leq \df_p(G)  \leq \dim H^1(G,\F_p) -
\dim H^2(G,\F_p),
\end{equation}
so
$$
\df(G)\leq \df_p(G) \leq \ochi_2(G,\F_p) +1.
$$
When $G$ is pro-$p$, we obtain from (\ref{def(G)}) that $
\df_p(G)= \ochi_2(G,\F_p) -1$; by \cite[Corollary 5.5]{Lu}
$d_G(\bar R)=\max \{ d(G),\dim H^2(G,\F_p) \}$ from where we
deduce that
 $$\df(G)=\min\{0,\df_p(G)\}.$$


\begin{lemma} \label{propquot}
Let  $G$ be  a profinite group and $G_{[p]}$ its maximal pro-$p$
quotient, then  $\ochi_2(G,\F_p)\le \ochi_2(G_{[p]},\F_p)$.
\end{lemma}

\begin{proof} Let $N$ be the kernel of the natural map $G\to G_{[p]}$. Then $H^1(N,\F_p)=0$. Thus,
the lemma follows from Corollary \ref{consLHS}(3).
\end{proof}

\begin{lemma}\label{propSylow}
Let $G$ be a profinite group. If the Sylow pro-$p$ subgroup $G_p$
of $G$ is not cyclic, we can find an open subgroup $U$ of $G$ such
that $\ochi_1(U,\F_p)\geq 1$.
\end{lemma}

\begin{proof}
Let $G=U_1>U_2>\ldots$ be a chain of open normal subgroups such
that $\cap_i U_i=G_p$. Thus we have that
$$G_p=\displaystyle \underleftarrow{\lim}\  U_i ~ \mbox{ and } ~H^1(G_p,\F_p) = \underrightarrow{\lim}\ H^1(U_i,\F_p).$$
Since $G_p$ is not cyclic, we have $\dim H^1(G_p,\F_p)
>1 $ and so $\dim  H^1(U_i,\F_p) >1$ for some $i$. Therefore
$\ochi_1(U_i,\F_p)=  \dim H^1(U_i,\F_p) -1 \geq 1 $.
\end{proof}

\begin{lemma} \label{cof}
Let $G$ be a finitely presented profinite group, $N$ a closed
subgroup and $H$ an open subgroup containing $N$.   Then
$$
\df_{p}(H,N)-1\ge [G:H] (\df_{p}(G,N)-1)
$$
and
$$
d_{p}(H,N)-1\le [G:H] (d_{p}(G,N)-1).
$$
\end{lemma}

\begin{proof}  Let  $M$ be a finite $\F_p[H]$-module. By Shapiro's lemma
$ H^i(H,M)=H^i(G,Coind^G_H(M))$ for all $ i$, so we obtain that
$$
\ochi_i(H,M)=[G:H]\ochi_i(G,Coind^G_H(M)) .
$$
Thus,
$$
\begin{array}{rcl}
\df_{p}(H,N)-1 & = & \min_{M \in
\mathcal{M}_p(N)}\{\ochi_2(H,M)\} \medskip \\
 & = & [G:H]\min_{M \in \mathcal{M}_p(N)}\{\ochi_2(G,Coind^G_H(M))\} \medskip \\
 & \ge & [G:H] (\df_{p}(G,N)-1)
 \end{array}
$$
and $ d_{p}(H,N)-1 =
\max_{M\in\mathcal{M}_p(N)}\{\ochi_1(H,M)\}\le [G:H]
(d_{p}(G,N)-1).$
\end{proof}
We shall finish the section with two general
technical results about the numerical invariants introduced here.

\begin{prop}\label{restr}
Let $G_1$ be a finitely generated profinite group, $N$ a normal
subgroup, $ G_2$ an open normal subgroup of index a power of $p$
in $G_1$ containing $N$ and $M$ a non zero finite
$\F_p[[G_1]]$-module. Denote by $\beta_i$ ($i=1,2$) the
restrictions maps $H^2(G_i,M)\to H^2(N,M)$ respectively. Then
$$
\dim(\im \beta_2) \ge  ( \ochi_2(G_1,M)\dim(M) +\dim(\im
\beta_1))[G_1:G_2]-\ochi_2(G_2,M)\dim(M)
$$
and, in particular,
$$
\ochi_1(G_2,M)\ge [G_1:G_2] \left(\frac{\dim(\im
\beta_1)}{\dim(M)}+\ochi_2(G_1,M)\right).
$$
\end{prop}
\begin{proof}
 Let
$$1\to R\to F_1 \stackrel{\phi}{\to} G_1\to 1$$ be a   presentation for
$G_1$ with $d(F_1)=d(G_1)$. We have the following presentation for
$G_2$:
$$1\to R\to F_2\stackrel{\phi}{\to} {G_2}\to 1,$$
where $F_2=\phi^{-1}({G_2})$.

From Corollary \ref{consLHS}(1), we obtain the following two exact
sequences ($i=1,2$):
\begin{equation}\label{posl1}
0\to H^1(G_i,M)\to H^1(F_i,M)\to H^1(R,M)^{{G_i}}\to  H^2( G_i,M
)\to 0.
\end{equation}
Note that $\ochi_1(F_i,M)=d(F_i)-1$ (see Example \ref{xi(F)}).
Therefore, using the Schreier formula for $F_i$ and the equality
$H^0(F_i,M)=H^0(G_i,M)$ one has
\begin{equation}\label{h1rm}
\begin{array}{c} \dim H^1(R,M)^{G_i}  =  \dim H^1(F_i,M)-\ochi_2(G_i,M)\dim(M) - \\
  \\
-\dim H^0(G_i,M)=
(\ochi_1(F_1,M)[G_1:G_i]-\ochi_2(G_i,M))\dim(M).\end{array}
 \end{equation}

Let $\alpha_i$ be the composition of the map $H^1(R,M)^{{G_i}}\to
H^2({G_i},M)$ from (\ref{posl1}) and the restriction map
$\beta_i\colon H^2({G_i},M)\to H^2(N,M)$. Put
$$
M_i=\ker \alpha_i.
$$
Since the following diagram is commutative
$$
\begin{array}{ccc}
   H^1(R,M)^{{G_1}} & \stackrel{\alpha_1}{\longrightarrow}
 & H^2({N},M)\\
  \curvearrowleft && \parallel\\
 H^1(R,M)^{{G_2}}&\stackrel{\alpha_{2}}{\longrightarrow}&H^2({N},M),
\end{array}
$$
we obtain that $M_{2}^{G_1}\le M_1$. Recall that
$$
\begin{array}{lll}
\dim(M_1) & = &  \dim H^1(R,M)^{G_1}-\dim(\im \alpha_1)\\
& &\\
&=& \dim H^1(R,M)^{G_1}-\dim(\im \beta_1)\\
&&\\
&=&(\ochi_1(F_1,M) -\ochi_2(G_1,M))\dim(M)-\dim(\im \beta_1).
\end{array}
$$
Since $G_1/G_2$ is a finite $p$-group, we have $\dim(M_{2})  \le
\dim (M_2^{G_1})[G_1:G_2]$ (see page 7 in \cite{Khukhro} for
$G_1/G_2$ cyclic, the general case follows by an obvious induction
on the order of $G_1/G_2$). Hence we get
$$
\begin{array}{c}
\dim(M_{2})  \le \dim (M_2^{G_1})[G_1:G_2]\le \dim(M_1)[G_1:G_2]=\\ \\
= ( (\ochi_1(F_1,M) -\ochi_2(G_1,M))\dim(M)-\dim(\im
\beta_1))[G_1:G_2]).
\end{array}
$$
Thus, combining this with (\ref{h1rm}) we get
$$
\begin{array}{c}
\dim(\im \beta_2)=\dim(\im \alpha_2) = \dim H^1(R,M)^{G_2}-\dim(M_2)\medskip \\
\ge ( \ochi_2(G_1,M)\dim(M)+\dim (\im \beta_1))
[G_1:G_2]-\ochi_2(G_2,M) \dim(M).
\end{array} $$
Since
$$
\noindent\begin{array}{l}
\ochi_1(G_2,M)=\frac{\ochi_2(G_2,M)\dim(M)+\dim
H^2(G_2,M)}{\dim(M)} \mbox{ and } \dim H^2(G_2,M)\geq \dim{\rm
Im}(\beta_2),
\end{array}
$$
we obtain
$$
\begin{array}{rl}
\ochi_1(G_2,M) & \ge \frac{\dim(\im
\beta_2)+\ochi_2(G_2,M)\dim(M)}{\dim(M)}\medskip \\
 & \ge  (\ochi_2(G_1,M)+\frac{\dim(\im \beta_1)}{\dim(M)})[G_1:G_2].
 \end{array}
 $$
\qed \end{proof}

\begin{prop} \label{fgns} Let $G$ be a finitely generated profinite
group and $N$ a normal subgroup such that $\ochi_1(N,\F_p)$ is
non-negative and finite. Then there exists an open subgroup $V$ of
$G$ containing $N$ such that for any open subgroup $U$ of $V$
containing $N$
$$
\begin{array}{lll}
 \ochi_2(U,\F_p)& \le &
-\ochi_1(U/N,\F_p)\ochi_1(N,\F_p)-\dim H^2(U/N,\F_p)\\ &&\\
&=&-\left[\ochi_1(U,\F_p)-\ochi_1(N,\F_p)-1\right]\ochi_1(N,\F_p)-\dim
H^2(U/N,\F_p).
\end{array}
$$
\end{prop}
\begin{proof}
Since $N^p[N,N]$ is open in $N$ there exists an open subgroup $J$
of $G$ such that $J\cap N=N^p[N,N]$. Put $V=JN$ and let $U$ be an
open subgroup of $V$. Then
$\ochi_1(U,\F_p)=\ochi_1(U/N,\F_p)+\ochi_1(N,\F_p)+1$. Thus using
Corollary \ref{consLHS}(5), we obtain that
$$\begin{array}{lll}
\ochi_2(U,\F_p)& = & \ochi_1(U,\F_p)-\dim H^2(U,\F_p)\\ & & \\
&\le &
(\ochi_1(U/N,\F_p)+\ochi_1(N,\F_p)+1)- \\
& & \\
& & ~~~ -(\ochi_1(U/N,\F_p)+1)(\ochi_1(N,\F_p)+1) -\dim H^2(U/N,\F_p)\\
& &
\\
& = & -\ochi_1(U/N,\F_p)\ochi_1(N,\F_p)-\dim H^2(U/N,\F_p)\\ & &\\
&=&-[\ochi_1(U,\F_p)-\ochi_1(N,\F_p)-1]\ochi_1(N,\F_p)-\dim
H^2(U/N,\F_p).\end{array}$$
\end{proof}

\subsection{The number of generators of modules over a profinite
group}\label{number generators}

In this subsection we describe a  way to calculate the number
of generators of a profinite $G$-module. This will be used several
times in the paper. As an application we obtain a characterization
of a profinite group to be of type $p$-$FP_m$ similar to
Lubotzky's characterization of a profinite group to have finite
deficiency.

For any irreducible $\Z_p[[G]]$-module $M$, denote by $I_M$ the
annihilator of $M$ in $\Z_p[[G]]$. If $K$ is a $\Z_p[[G]]$-module,
then $K/I_MK\cong  (\Z_p[[G]]/I_M)\widehat{\otimes}_{\Z_p[[G]]}K$
is the maximal quotient of $K$ isomorphic to a direct sum of
copies of $M$. Thus, $\Z_p[[G]]/I_M$ {is the maximal cyclic
$\Z_p[[G]]$-module isomorphic to a direct sum of copies of} $M$.

Note that the Jacobson radical $J(K)$ of a $\Z_p[[G]]$-module $K$ is equal to the
intersection of $I_MK$ and so $$K/J(K)\cong \prod_ {M \text{\ is
irreducible}} K/I_MK.$$ Thus

$$
d_G(K)=d_G(K/J(K))=\max_{M \text{\ is irreducible}}d_G(K/I_MK),
$$
and so we obtain that
$$
d_G(K)=\max_{M \text{\ is irreducible}}\left
\lceil\frac{\dim(K/I_MK)}{\dim (\Z_p[[G]]/I_M)}\right \rceil.
$$
Since $K/I_MK$ and $\Z_p[[G]]/I_M$ are direct sums of copies of
$M$, we conclude that
$$\left
\lceil\frac{\dim(K/I_MK)}{\dim (\Z_p[[G]]/I_M)}\right \rceil=\left
\lceil\frac{\dim \Hom_{\Z_p[[G]]}(K/I_MK, M)}{\dim
\Hom_{\Z_p[[G]]}(\Z_p[[G]]/I_M,M)}\right \rceil.$$  Note that
$$\Hom_{\Z_p[[G]]}(K/I_MK, M)\cong \Hom_{\Z_p[[G]]}(K, M)$$ and
$$\dim \Hom_{\Z_p[[G]]}(\Z_p[[G]]/I_M,M)=\dim
\Hom_{\Z_p[[G]]}(\Z_p[[G]],M)=\dim M.$$ Thus, we conclude that
\begin{equation}\label{numgen}
d_G(K)=\max_{M \text{\ is
irreducible}}\left \lceil\frac{\dim \Hom_{\Z_p[[G]]}(K,M)}{\dim
M}\right \rceil.
\end{equation}

The next theorem is inspired by a theorem of Lubotzky
\cite[Theorem 0.3]{Lu} that says that a finitely generated
profinite group is finitely presented if and only if there exists
$C$ such that $\dim H^2(G,M)\le C\dim M$ for any irreducible $\hat
\Z[[G]]$-module $M$.
\begin{theorem}\label{critfp} Let $G$ be a profinite group. Then
$G$ is of type $p$-$FP_m$ if and only if there exists a constant
$C$ such that $\dim H^i(G,M)\le C \dim M$ for any irreducible $
\Z_p[[G]]$-module $M$ and any $0\le i\le m$.
\end{theorem}
\begin{proof} We prove the theorem by induction on $m$. The case $m=0$ is
trivial. Assume that theorem holds for $m-1$.

Suppose, first, that $G$ is of type $p$-$FP_m$. By induction,
there exists $C^\prime$ such that $\dim H^i(G,M)\le C^\prime \dim
M$ for any irreducible $ \Z_p[[G]]$-module $M$ and any $1\le i\le
m-1$. Since $G$ is of type $p$-$FP_m$, there exists an exact
sequence of finitely generated projective modules
$$
\mathcal{R}:\ \  P_m\to P_{m-1}\to \ldots \to P_0\to \Z_p\to 0.
$$
Let $M$ be  an irreducible $\Z_p[[G]]$-module.  If we apply
$\Hom_{\Z_p[[G]]} (-,M)$ to the complex $\mathcal{R}_{del}$ obtained by suppressing $\Z_p$, we obtain the complex $\Hom_{\Z_p[[G]]} (\mathcal{R}_{del},M):$
$$ 0\to \Hom_{\Z_p[[G]]}(P_0,M)\to \ldots\to  \Hom_{\Z_p[[G]]}(P_{m-1},M)\to^\phi \Hom_{\Z_p[[G]]}(P_m,M).$$
The cohomology groups $H^{i}(G,M)$ for $i\le m-1$ are the cohomology groups of this complex and   $H^m(G,M)$ is a subgroup of  $\Hom_{\Z_p[[G]]}(P_m,M)/\im \phi$. Hence
\begin{equation}\label{hom}
\sum_{i=0}^m (-1)^{m-i}\dim \Hom_{\Z_p[[G]]}(P_i,M)\ge
\sum_{i=0}^m (-1)^{m-i}\dim H^i(G,M),
\end{equation}

Therefore
$$
\begin{array}{lll}
\dim H^m(G,M)&\le& \sum_{i=0}^m \dim \Hom_{\Z_p[[G]]}(P_i,M)+
\sum_{i=0}^{m-1} \dim H^i(G,M)\\&&\\
&\le& (\sum_{i=0}^m d_G(P_i)+mC^\prime)\dim M.
\end{array}
$$
Thus we may put $C=\sum_{i=0}^m d_G(P_i)+mC^\prime$.

Suppose now that there exists a constant $C$ such that $\dim
H^i(G,M)\le C \dim M$ for any irreducible $ \Z_p[[G]]$-module $M$
and any $0\le i\le m$. By inductive assumption, there exists an
exact sequence
$$ \mathcal{R}:\ \  0\to A\to P_{m-1}\to \ldots \to P_0\to \Z_p\to 0$$
with $P_i$ finitely generated projective for $0\le i\le m-1$. We
want to show that $A$ is finitely generated, since then we can
cover it by a finitely generated free module. Let $M$ be  an irreducible $\Z_p[[G]]$-module.  If we apply
$\Hom_{\Z_p[[G]]} (-,M)$ to the complex $\mathcal{R}_{del}$ obtained by suppressing $\Z_p$, we obtain the complex $\Hom_{\Z_p[[G]]} (\mathcal{R}_{del},M):$
$$ 0\to \Hom_{\Z_p[[G]]}(P_0,M)\to \ldots\to  \Hom_{\Z_p[[G]]}(P_{m-1},M)\to^\phi \Hom_{\Z_p[[G]]}(A,M)\to 0.$$
The cohomology groups $H^{i}(G,M)$ for $i\le m-1$ are the cohomology groups of this complex  and  $\Hom_{\Z_p[[G]]}(A,M)/\im \phi$  is a subgroup of  $H^m(G,M)$ .  Hence
$$
\begin{array}{lll}
\dim \Hom_{\Z_p[[G]])}(A,M) & \le & -\sum_{i=0}^{m-1}
(-1)^{m-i}\dim \Hom_{\Z_p[[G]]}(P_i,M)\\&&\\&+& \sum_{i=0}^m \dim
(-1)^{m-i} H^i(G,M)\\&&\\ &\le&(\sum_{i=0}^{m-1}
d_G(P_i)+(m+1)C)\dim M.
\end{array}
$$
Therefore, by (\ref{numgen}), $A$ is finitely
generated.
\end{proof}

 We will need the following application of the previous theorem.
 \begin{corollary}\label{pf2} Let $G$ be a profinite group of type $p$-$FP_2$
 and $N$ a   normal subgroup of $G$. Assume that $N$ is finitely generated as a normal subgroup. Then $G/N$ is of type $p$-$FP_2$.
 \end{corollary}
 \begin{proof} Let $M$ be an irreducible $\Z_p[[G/N]]$-module. From Corollary \ref{consLHS}(1), we
 obtain that
 $$\dim H^1(G/N,M)\le \dim H^1(G,M)$$ and
 $$
 \dim H^2(G/N,M) \le \dim H^1(N,M)^{G/N}+\dim H^2(G,M).$$
Since $M^N=M$ we have that $H^1(N,M)=\Hom (N,M)$ (the set of
continuous homomorphisms from $N$ to $M$).  Recall that the action
of $G$ on $\Hom (N,M)$ is defined in the following way: for $x\in
G$ and $f\in \Hom (N,M)$
$$(xf)(y)=xf(y^x),\ \ y\in N.$$
Thus if $f\in \Hom (N,M)^{G/N}$, $f(y^x)=x^{-1}f(y)$. Let
$n_1,\ldots,n_l$ be generators of $N$ as a normal subgroup. Then
$N$ is generated by $\{n_i^x:x\in G, 1\le i\le   l\}$. Therefore
any element of $\Hom (N,M)^{G/N}$ is completely determined by the
values of $\{f(n_i)\}$. Thus, $\dim \Hom (N,M)^{G/N}\le l\dim M$.
This implies that
 $$
 \dim H^2(G/N,M) \le  l\dim M+\dim H^2(G,M).$$  Since $G$ is of type $p$-$FP_2$, the previous
 theorem implies that $G/N$ is also of type $p$-$FP_2$.
 \end{proof}

\section{Profinite groups of positive deficiency}\label{Normal subgroups of profinite groups of positive deficiency}

This section consists of main results on profinite groups of
positive deficiency. We  divide our results in subsections by the
reverse order on deficiency.

\subsection{Groups of deficiency $\geq 2$}
\begin{prop}\label{def2}
Let $G$ be a finitely generated profinite  group, $N$ a normal
subgroup  and $H$ and $J$ two open subgroups containing $N$. If
$\df_{p}(H,N)\ge 2 $, then
$$\frac{d_{p}(J,N)-1 }{[G:J]}\ge \frac{1}{[G:H]}.$$
\end{prop}
\begin{proof} Using Lemma \ref{cof}, we obtain that
$$\begin{array}{lll}d_{p}((J\cap H),N) & \ge & \df_{p}((J\cap H),N)  \ge     [H:(J\cap H) ](   \df_{p}(H,N)-1)+1\\
&&\\
& \ge  &  [H:(J\cap H)]+1
\end{array}
 $$
Hence, by Lemma \ref{cof},
$$d_{p}(J,N)-1\ge \frac{d_{p}((J\cap
H),N)-1}{[J:J\cap H]}\ge \frac{[H:(J\cap H)]}{[J:J\cap H]}\ge
\frac{[G:J]}{[G:H]}  .$$
\end{proof}

\begin{prop}\label{largenorm} Let $G$ be a finitely generated profinite group and $N$
a normal subgroup of infinite index such that $H^1(N,\F_p)\ne 0$.
If $\df_{p}(G,N)\ge 2$, then  $H^1(N,\F_p)$ is infinite.
\end{prop}

\begin{proof}
If  $H^1(N,\F_p)$ is  finite, then by Proposition \ref{fgns} there
exists an open subgroup $U$ containing $N$ such that
$$
\begin{array}{rcl}
\ochi_2(U,\F_p) &\le & -\ochi_1(U/N,\F_p)\ochi_1(N,\F_p)-\dim
H^2(U/N,\F_p).
\end{array}
 $$
Since by Proposition \ref{def2} we have $\ochi_1(U/N,\F_p)$ non
negative, it follows that $\ochi_2(U,\F_p)\le 0$. Thus $\df_p(U,N)
= \min_{M\in \mathcal{M}_p(N)}\{1+ \overline{\chi}_2(U,M)\}\leq
1$. By  Lemma \ref{cof},  $\df_p(G,N)\geq 2$ implies
$\df_p(U,N)\geq 2$, a contradiction. \end{proof}

Theorem 8.6.5 in \cite{RZ} originally proved by Melnikov states
that a non-trivial normal subgroup of a non-cyclic free profinite group of
infinite index is infinitely generated. The next corollary extends
this to profinite groups of deficiency $\geq 2$.

\begin{corollary}\label{def2ns} Let $G$ be a finitely generated profinite group and $N$
a    normal subgroup of infinite index such that $p$ divides
$|N|$. If $\df_p(G)\ge 2$, then some open subgroup of $N$ has
infinite $p$-abelianization. In particular, any non-trivial normal subgroup of
a profinite group of deficiency $\geq 2$ that has infinite index is infinitely generated.
\end{corollary}
\begin{proof}
Find an open subgroup $U$ of $G$ such that $H^1(U\cap N,\F_p)\ne
0$ and apply the previous proposition.
\end{proof}

\subsection{Groups of deficiency 1}
The main tool of this subsection is the  following result.

\begin{theorem}\label{def1}
Let $G$ be a finitely generated profinite group, $K\le N$ two
normal subgroups of $G$ such that $|G/N|_p$ is  infinite and
$\df_{p}(G,K)\ge 1$. Let $M$ be a non zero finite
$\F_p[[G]]$-module on which $K$ acts trivially.
  Suppose that
$$
\inf \left\{\frac{d_{p}(H,K)-1 }{[G:H]_p}\  | \ N<H\le_0 G
\right\}=0,
$$
where $[G:H]_p$ is the greatest power of $p$ dividing $[G:H]$.
Then
\begin{enumerate}
\item[\em{(1)}]
 $\df_{p}(U,K)= 1$  for any open subgroup $U$ containing $N$   and
 \item[\em{(2)}] $H^2(N,M)=\{0\}$. \end{enumerate}
\end{theorem}
\begin{proof} Let $U$ be an open subgroup of $G$ containing $N$.
Then by Lemma \ref{cof}, $\df_{p}(U,K)\ge 1$      and by
 Proposition \ref{def2},
$\df_{p}(U,K)\le 1$.  Thus, $\df_{p}(U,K)=1$.

Now, by way of contradiction let us assume that $H^2(N,M)\ne\{
0\}$. Let $G=H_1>H_2>\ldots$  be a chain of open normal subgroups
such that $\cap_i H_i=N$. Thus we have that
$$N=\displaystyle \underleftarrow{\lim}\  H_i.$$
Hence $H^2(N,M)=\underrightarrow{\lim}\ H^2(H_i,M)$. Since
$H^2(N,M)\ne 0$, we obtain that there exists $j$ such that the
image of the restriction map $\beta: H^2(H_j,M)\to H^2(N,M)$ is
not zero.

Let $H$ be an open normal subgroup of $G$ contained in $H_j$ which
contains $N$ and $P$ a subgroup of $H_j$ containing $H$ such that
$P/H$ is a $p$-Sylow subgroup of $H_j/H$. By (1) we have
$\df_{p}(P,K)=1$, so $\ochi_2(P,M)\ge 0$. Therefore by Proposition
\ref{restr},
$$
\ochi_1(H,M) \ge   ( \ochi_2(P,M) + \frac{\dim(\im
\beta)}{\dim(M)})[P:H] \ge \frac{[P:H]}{
\dim(M)}=\frac{[G:H]_p}{[G:H_j]_p\dim(M)}.
$$
Hence $d_{p}(H,K)\ge 1+ \frac{[G:H]_p}{[G:H_j]_p\dim(M)}$. Now, if
$H$ is an arbitrary open subgroup containing $N$, applying Lemma
\ref{cof} for $H$ and $H \cap H_j$ and the latter  inequality for
$H\cap H_j$, we get
$$
\frac{d_{p}(H,K)-1}{[G:H]_p} \ge \frac{d_{p}((H\cap H_j),K)
)-1}{[H:(H\cap H_j )][G:H]_p}\ge
 \frac{1}{[G:H_j]_p[H:(H\cap H_j )]_{p^\prime}\dim(M)}.
 $$
 \medskip
 Since $H_j$ is fixed, $\frac{1}{ [H:(H\cap H_j )]_{p^\prime}}$ has positive lower bound that gives a contradiction   with the hypothesis. Hence, $H^2(N,M)=0$.
\qed \end{proof}

\begin{corollary}\label{def1prop}
Let $G$ be a finitely generated pro-$p$ group with $\df_p(G)>0$ and $N$ a normal
subgroup of infinite index.
  Suppose that
$$\inf \left\{\frac{d(H)-1 }{[G:H]}\  | \ N<H\le_0
G \right\}=0.$$ Then $N$ is a free pro-$p$ group.
 \end{corollary}
 \begin{proof} The previous theorem implies that $H^2(N,\F_p)=0$.
 \end{proof}
\subsection{Finitely generated normal subgroups of profinite
groups of deficiency 1}

We need the following criterion for a group of positive deficiency
to have cohomological $p$-dimension 2.
\begin{prop} \label{critcd2} Let $G$ be a finitely generated profinite group with $\df_p(G)= 1$.
Suppose that for any open subgroup $V$ of $G$ there exist an open
subgroup $U$ of $V$ such that $\ochi_2(U,\F_p)=0$. Then
$\cd_p(G)\le 2$.
\end{prop}
\begin{proof}
  Since $\df_p(G)=1$, there exists
an exact sequence of  modules $$  \mathcal{R}: 0\to M\to
\F_p[[G]]^{d-1}\to \F_p[[G]]^d\to\F_p[[G]]\to\F_p\to 0,$$ where
$d=d(G)$ (see (\ref{presentation})). Let $V$ be an open subgroup
of $G$ and let $U\le_o V$ be such that $\ochi_2(U,\F_p)=0$.
Applying the functor $\F_p \otimes_{\F_p[[U]]}-$ to $\mathcal{R}$
we obtain the complex
$\mathcal{R}_U=\F_p \otimes_{\F_p[[U]]} \mathcal{R}$ given by
$$
0\to \F_p \otimes_{\F_p[[U]]} M \stackrel{h}{\to}
\F_p[[G/U]]^{d-1}\stackrel{g}{\to}
\F_p[[G/U]]^d\stackrel{f}{\to}\F_p[[G/U]]\to\F_p\to 0.
$$
Let $n=|G/U|$. Counting $\F_p$-dimension one gets
$$
\begin{array}{rcl}
\dim H_1(U,\F_p) & = & \dim H_1(\mathcal{R}_U) \medskip\\
 & = & \dim(\ker f)- \dim(\im g) \medskip\\
 & = & nd - n + 1 - [n(d-1)-(\dim(\im h)+ \dim H_2(\mathcal{R}_U)] \medskip \\
 & = & 1 + \dim(\im h) + \dim H_1(U,\F_p) - 1.
 \end{array}
 $$
It follows that $\im h=0$ and so
 $M=0$. Thus, $\cd_p(G)\le 2$ (cf. Proposition \ref{projdim}).
\end{proof}

\begin{remark} \rm Note that the hypothesis from the previous
proposition are equivalent to $\ochi_2(G,M)=0$ for any non zero
finite $\F_p[[G]]$-module $M$.
\end{remark}

Now we are ready to prove Theorem \ref{tmaintro}.

\begin{theorem}\label{tma} Let $p$ be a prime.
Let $G$ be a finitely generated profinite group with $\df_p(G)\ge
1$ and $N$ a finitely generated normal subgroup such that
$|G/N|_p$ is  infinite and $p$ divides $|N|$. Then $\df_p(G)=1$
and either the $p$-Sylow subgroup of $G/N$ is virtually cyclic
or the $p$-Sylow subgroup of $N$ is cyclic. Moreover,
$\cd_p(G)=2$, $\cd_p(N)=1$ and $\vcd_p(G/N)=1$.
\end{theorem}

\begin{proof} First observe that Corollary \ref{def2ns} implies
$\df_p(G)=1$.  Note also that  by Proposition \ref{fgns} for any
open subgroup $J$ of $G$ such that $\ochi_1(N\cap J,\F_p)\geq 0$
there exists an open subgroup $V$ of $J$ containing $J\cap N$ such
that for any open subgroup $U$ of
 $V$ containing $J\cap N$,
 \begin{equation} \label{t-U/NJ}
 \ochi_2(U,\F_p) \le -\ochi_1(U/(N\cap J),\F_p)\ochi_1(N\cap J,\F_p)-\dim H^2(U/N,\F_p.)
 \end{equation}

{\it Claim 1.} Suppose that a $p$-Sylow subgroup of $N$ is not
cyclic. Then

\begin{enumerate}
\item[(i)]  $G/N$ is virtually cyclic. \item[(ii)] $\cd_p(N)=1$
and $\vcd_p(G)=2$.
\end{enumerate}

(i) By Lemma \ref{propSylow} we can find an open subgroup $J$ of
$G$ such that $\ochi_1(J\cap N,\F_p)\geq 1$. Then from Equation
(\ref{t-U/NJ}) we get $\ochi_1(U/(N\cap J),\F_p)\le 0$, because
$\ochi_2(U,\F_p) \ge \df_p(U)-1\ge 0$.
 This means that the $p$-Sylow subgroup of $G/N$ is virtually cyclic.

 \

 (ii)
Let $W$ be an open subgroup of $G$.   Since $W\cap N$ is finitely
generated and the $p$-Sylow subgroup of $W/(W\cap N)$ is
virtually cyclic, $W$ and $W\cap N$ satisfy the hypothesis of
Theorem \ref{def1}. Hence we obtain that $H^2(W\cap N,\F_p)=0$ and
by Proposition \ref{projdim}, $\cd_p( N)\le 1$. Since $p$ divides
$|N|$, we obtain  that $\cd_p(N)=1$. Hence $\vcd_p(G)=2$.

\

{\it Claim 2.}  Suppose that the $p$-Sylow subgroup of $N$ is
cyclic. Then $\vcd_p(G/N)=1$ and $\vcd_p(G)=2$.

\

We can find an open subgroup $J$ of $G$ such that $\ochi_1(J\cap
N,\F_p)= 0$. Applying Equation (\ref{t-U/NJ}) we deduce that for
any open subgroup $U$ of
 $J$ containing $J\cap N$, $\dim H^2({U/N\cap J},\F_p)= 0$. Hence,
 $\cd_p(V/(J\cap N))=1$, so $\vcd_p(G/N)=1$ and again $\vcd_p(G)=2$.

 \

 {\it Claim 3.}  $\cd_p(G)= 2$.

\

Let $V$ be an open subgroup of $G$. Since $|N|_p$ and $|G/N|_p$
are infinite, we can find an open subgroup $U$ of $V$ such that
$\ochi_1(U\cap N,\F_p)$ and $\ochi_1(U/U \cap N,\F_p)$ are
non-negative. Then putting in equation \ref{t-U/NJ} $J=G$ and
using the equality $\ochi_2(U/N,\F_p)=\ochi_1(U/N,\F_p)-\dim
H^2(U/N,\F_p)$, we get $\ochi_2(U,\F_p)\le 0$. Hence, by
Proposition \ref{critcd2}, $\cd_p(G)\le 2$.
\end{proof}

Theorem \ref{tma} is the profinite version of \cite[Theorem
4]{HS}, where $G$ was assumed pro-$p$.

\begin{corollary}\label{profinite}
Let $G$ be a finitely generated profinite group of positive
deficiency and $N$ a finitely generated normal subgroup of $G$
such that for every prime $p$ dividing $|N|$ a $p$-Sylow subgroup
$(G/N)_p$ is infinite. Then $N$ is projective.  Moreover,
 if  none of the non-trivial  $p$-Sylow subgroups $(G/N)_p$ is
virtually cyclic then $N$ is
solvable of type $\Z_\pi\rtimes \Z_\rho$ where
$\pi$ and $\rho$ are disjoint sets of primes.

\end{corollary}

\begin{proof} By the previous theorem
$N$ is projective. By \cite[Exercise 2.3.18]{RZ}, a
profinite group with cyclic Sylow subgroups is of type
$\Z_\pi\rtimes \Z_\rho$ where $\pi$ and $\rho$ are disjoint sets
of primes.
\end{proof}

\begin{corollary}\label{cora}  Let $G$ be a finitely generated profinite
group of positive deficiency whose commutator subgroup $[G,G]$ is
 finitely generated. Then $\df(G)=1$ and $[G,G]$ is projective. Moreover,
 $\cd(G)=2$
unless $G=\widehat{\mathbb{Z}}$. \end{corollary}

\begin{proof}
First let us suppose that $ G$ is   abelian. We want to show that
 $G\cong \hat{\mathbb{Z}}\times \Z_\pi$ for some set of primes $\pi$ (possibly empty).

  If $d(G)=1$ then positive deficiency means
$0$ relations, so the result is obvious in this case.

Let $G_{[p]}$ be a maximal pro-$p$ quotient of $G$. Then
$d(G)=d(G_{[p]})$ for some $p$, and so $\df_p(G_{[p]})\ge \df(G)$
 because any presentation of $G$ serves as a
presentation for $G_{[p]}$ as a pro-$p$ group. Since
$\df_p(G_{[p]})=\dim H_1(G_{[p]},\mathbb{F}_p)-\dim
H_2(G_{[p]},\mathbb{F}_p)$ we have $\df_p(G_{[p]})\leq 0$ for
$d(G)>2$. Therefore $\df(G)\leq 0$ for $d(G)>2$.

Suppose $d(G)=2$. It suffices to prove  that $G_{[p]}$ is
non-trivial for every $p$. But this is clear since otherwise
$0\leq \df(G)\leq \df_p(G)\leq \df_p(G_{[p]})=0$, a contradiction.

\medskip
Suppose now that $G$ is not abelian. Let $G_{[p]}$ denote again the
maximal pro-$p$ quotient of $G$. Then   $\df_p(G_{[p]})\ge \df(G) >0$
and so, by \cite[Window 5, Sec.1, Lemma 3]{LS}, $G_{[p]}$ has infinite
abelianization. Hence $G$ has $\mathbb{Z}_p$ as an epimorphic
image for every $p$ and therefore has $\widehat{\mathbb{Z}}$ as a
quotient.  Then Theorem \ref{tma} implies that  $\df(G)=1$, $\cd(G)=2$ and $[G,G]$ is projective.
\end{proof}

\begin{remark} \rm The groups considered in   Subsection \ref{Ascending HNN-extensions} show that $[G,G]$ does not have to be free profinite.\end{remark}

\subsection{Pro-$p$ groups of subexponential subgroup growth}
Let $G$ be a  profinite group. Denote by $a_n(G)$ the number of
open subgroups of $G$ of index $n$. If $G$ is finitely generated
then $a_n(G)$ is finite for all $n$. We say that a group $G$ is of
{\bf subexponential} subgroup growth if $\limsup_{n\to \infty}
a_n(G)^{1/n}=1$.  The following characterization of pro-$p$ groups
of subexponential subgroup growth is given by Lackenby.
\begin{prop}(\cite[Theorem 1.7]{Lac})\label{subexp}
Let $G$ be a finitely generated pro-$p$ group. Then $G$ is of
subexponential subgroup growth if and only if $$\limsup_{[G:U]\to
\infty}\frac{d(U)}{[G:U]}=0.$$
\end{prop}
For example, since $p$-adic analytic profinite groups have finite
rank, they  are of subexponential subgroup growth (in fact, they
are of polynomial subgroup growth).

\begin{lemma}\label{wrarg} Let $G$ be a finitely generated pro-$p$ group of subexponential subgroup growth and $N$ a normal subgroup of $G$ such that $G/N$ is
virtually cyclic. Then $N$ is finitely generated.
\end{lemma}
\begin{proof}  Without loss of generality we may assume that $G/N$ is cyclic. If $N$ is not finitely generated, then $N/\Phi(N)$ is an infinite finitely
generated $\F_p[[G/N]]$-module.
Since $\F_p[[G/N]]$ is isomorphic to the ring of power series over
$\F_p$, the $\F_p[[G/N]]$-module $N/\Phi(N)$ is a direct sum of
cyclic modules. Hence there exists  a normal subgroup $M$ of $G$
such that $\Phi(N)\le M<N$ and $N/M\cong \F_p[[G/N]]$
 as $\F_p[[G/N]]$ modules. Therefore $G/M$ is isomorphic to the pro-$p$ wreath product $C_p\hat \wr \Z_p$, i.e. to the inverse limit of
 wreath products $C_p\wr \Z/p^n\Z$. But the last  group is of exponential subgroup growth, whence $G$ is of exponential subgroup growth, a contradiction.
\end{proof}

We believe that the following conjecture holds.
\begin{conjecture} Let $G$ be a finitely generated pro-$p$ group of subexponential
subgroup growth with $\ochi_2(G,\F_p)=0$. Then $G$ is  $\Z_p$ or
$\Z_p\rtimes \Z_p$.
\end{conjecture}
We can prove the following result.
\begin{theorem}
Let $G$ be a finitely generated pro-$p$ group of subexponential
subgroup growth. If $\ochi_2(G,\F_p)=0$, then $G$ is  (finitely
generated free pro-$p$) by cyclic and all finitely generated
subgroups of infinite index are free pro-$p$ groups.
\end{theorem}
\begin{proof}
Since $G$ is a pro-$p$ group $\df_p(G)=\ochi_2(G,\F_p)+1=1$. Hence
there exists a map of $G$ onto $\Z_p$. Let $N$ be the kernel of
this map. By Lemma \ref{wrarg}, $N$ is finitely generated.
Applying Corollary \ref{def1prop}, we obtain that $N$ is a
finitely generated  free pro-$p$ group. In particular $\cd (G)=2$
and $\ochi_2(U,\F_p)=0$ for all open subgroups.

\

{\it Claim} Let $V$ be an open subgroup of $G$ and $U$  an open
normal subgroup of $V$. Assume that the image of $H_2(U,\F_p)$ in
$H_2(V,\F_p)$ is not trivial. Then $\ochi_1(U,\F_p)\ge
\ochi_1(V,\F_p)-1+[V:U]$.

\

Let $d=d(V)$. Since $V$ is  a pro-$p$ of cohomological dimension 2
and  its $p$-deficiency is 1, we have the following exact sequence
of right modules:

\begin{equation}\label{seqG}
0\to \F_p[[V]]^{d-1}\to  \F_p[[V]]^{d}\to \F_p[[V]]\to \F_p\to 0.
\end{equation}

Applying the functor  $- \widehat{\otimes}_{\F_p[[U]]}\F_p$ we
obtain the complex
\begin{equation} \label{tensU} 0\to
\F_p[V/U]^{d-1}\stackrel{\alpha}{\longrightarrow}
\F_p[V/U]^{d}\stackrel{\beta}{\longrightarrow} \F_p [V/U] \to
\F_p\to 0,\end{equation} where $H_2(U,\F_p)\cong \ker \alpha$ and
$H_1(U,\F_p)\cong \ker \beta/\im \alpha $. Note that we can
calculate $H_1(V,\F_p)$ and $H_2(V,\F_p)$ either tensoring
(\ref{tensU}) with $-\otimes_{\F_p[V/U]}\F_p$ or tensoring
(\ref{seqG}) with $-\widehat{\otimes}_{\F_p[[V]]}\F_p$, once the
obtained complexes are isomorphic.

Let $T$ be a transversal of $U$ in $V$. Denote by $a$ the element
$\sum_{t\in T} t$ of $\F_p[V]$. Since $d=d(V)$, the rank of
$$(\F_p[V/U]^{d}/\im \alpha)\widehat{\otimes}_{\F_p[[V]]} \F_p\cong H_1(V,\F_p)$$ is
$d$. Thus, $\im \alpha$ is contained in the Jacobson radical of
$V$-module $\F_p[V/U]^{d}$ and so $a\im \alpha=0$. Thus
$a(\F_p[V/U]^{d-1})\le \ker \alpha$ which implies that  $(\ker
\alpha)^V$ has rank $d-1$.

Since the image of $H_2(U,\F_p)$ in $H_2(V,\F_p)$ is not trivial,
the rank of\\ $(\F_p[V/U]^{d-1}/\ker
\alpha)\widehat{\otimes}_{\F_p[[V]]}\F_p$ is at most $d-2$. Hence
$\ker \alpha$ is not in the Jacobson radical of $V$-module
$\F_p[V/U]^{d-1}$. Since any element outside the  Jacobson radical
of   $\F_p[V/U]^{d-1}$ generates a $V$-submodule isomorphic to $
\F_p[V/U]$, we conclude that the $\F_p[[V]]$-module $\ker \alpha$
contains a submodule $Z$ isomorphic to  $\F_p[V/U]$. Thus, since
$\dim Z^V=1$ we obtain that
$$\dim H_2(U,\F_p)=\dim \ker \alpha \ge \dim  Z+\dim (\ker \alpha)^V-\dim Z^V=|V/U|+d-2,$$ and so $$\ochi_1(U,\F_p)\ge
\dim H_2(U,\F_p)\ge \ochi_1(V,\F_p)-1+[V:U].$$ This proves Claim.

\

Now, let $K$ be a finitely generated subgroup of infinite index in
$G$. By way of contradiction assume $H^2(K,\F_p)\ne 0$. Then there
exists an open subgroup $V$ of $G$ containing $K$ such that for
any open subgroup $U$ of $V$ containing $K$ the image or
restriction map $H_2(K,\F_p)\to H_2(U,\F_p)$ is not trivial. Let
$U_0=V$, and let $U_{i+1}$ ($i=0,1,\ldots$) be  a subgroup of
index $p$ in $U_i$ containing $K$. Applying Claim, we obtain that
$\ochi_1(U_i,\F_p)\ge \ochi_1(U_0,\F_p)+(p-1)i$. Hence we can find
an open subgroup $U$ of $V$ containing $K$ such that
$\ochi_1(U,\F_p)-1=d(U)$ is arbitrary large and in particular,
there exists $K<U\le_o V$ such that $$d(U)-1+d(K)< \frac
{(d(U)-d(K))^2}4.$$ Let $N$ be a normal subgroup of $U$ generated
by $K$. Then $U/N$ does not satisfy the Golod-Shafarevich
inequality (see, for example, \cite[Interlude D]{DDMS}):
$$\dim H_2(U/N,\F_p)\le d(U)-1+d(K)< \frac {(d(U)-d(K))^2}4\le
\frac {(d(U/N))^2}4,$$ and so $N$ is of infinite index in $U$.
Since $N$ is not free pro-$p$, Corollary \ref{def1prop} and
Proposition \ref {subexp} imply that $U$ is not of subexponential
subgroup growth, a contradiction. \qed \end{proof}

\section{Poincar\'e duality groups of dimension 3}\label{PDdrei}

Let now $G$ be a profinite group of type $p$-$FP_\infty$ and $B$ a
profinite $\Z_p[[G]]$-module. Then
$B=\lim\limits_{\displaystyle\longleftarrow} B_j$, where each
$B_j$ is a finite discrete $p$-torsion $\Z_p[[G]]$-module, and so
$H^i(G, B_j)$ is finite for all $i$ and $j$. Thus we can define
the $i$th  cohomology of $G$ with coefficients in the profinite
module $B$ as the profinite group
$$
H^i(G,B)=\lim\limits_{\displaystyle\longleftarrow} H^i(G,B_j).
$$
Note that this definition coincides with the one given in
\cite[Thm. 3.7.2]{SW} where it is
$\mathrm{Ext}^i_{\Z_p[[G]]}(\Z_p,B)$ (see also \cite[Corollary
2.3.5]{NSW}).

As defined in \cite{SW}, a profinite group $G$ of type $p$-$FP_{\infty}$  is called a
\textbf{Poincar\'e duality group} at $p$ of dimension $n$  if   $cd_p(G)=n$ and
$$
\begin{array}{ll}
H^i(G,\Z_p[[G]])= 0, & \mbox{ if } i\not= n, \vspace{0.2cm}\\
H^n(G,\Z_p[[G]])\cong \Z_p & \mbox {(as abelian groups)}.
\end{array}
$$
\begin{remark} \rm In \cite[page 165]{NSW} a more general definition of  a
 {Poincar\'e duality group} at $p$ is given (without the assumption of $G$ to be of type $p$-$FP_{\infty}$). If $G$ is of type $p$-$FP_{\infty}$, then both definitions coincide. In this paper we always assume that a  {Poincar\'e duality group} at $p$ is of type $p$-$FP_{\infty}$.
 \end{remark}
We use the term profinite $PD^n$-group at $p$ for a Poincar\'e
duality profinite group $G$ at $p$ of dimension $n$. If $G$ is a
profinite group with $cd_p(G)< \infty$ and $U$ is an open subgroup
of $G$, then $G$ is a profinite $PD^n$-group at $p$ if and only if
$U$ is a profinite $PD^n$-group at $p$.

By a result of Lazard, compact $p$-adic analytic groups $G$ are
virtual Poincar\'e duality groups of dimension $n=\dim(G)$ at a
prime $p$ (\cite[Thm. 5.9.1]{SW}). The Demushkin pro-$p$ groups
are exactly the pro-$p$ $PD^2$-groups (\cite[I.4.5 Example
2]{Serre}) and $\Z_p$ is the only pro-$p$ $PD^1$-group
(\cite[Example 4.4.4]{SW}).

Let  $G$  be $PD^n$ at $p$ and
$I_p(G)=Hom_{\Z_p}(H^n(G,\Z_p[[G]]),   \Q_p/\Z_p)$ its dualizing
module.  Then $I_p(G)$ is isomorphic (as abelian group) to
$\Q_p/\Z_p$ (note that the action of $G$ on $I_p(G)$ is not always
trivial). For any finite $\Z_p[[G]]$-module $M$,
$M^*=\Hom(M,I_p(G))$ is called the  {\bf  dual} of $M$. The action
of $G$ on $M$ is given as
$$(gf)(m)=gf(g^{-1}m),\ g\in G,\  m\in M, f\in M^*.$$
It is clear that $M^{**}\cong M$. We say that $M^*$ is {\bf
self-dual} if $M\cong M^*$.

In this section we are interested in
$PD^3$-groups at $p$.  The important consequences of the
Poincar\'e duality are recollected in the following proposition.

\begin{prop}\label{proppd3} Let $G$ be a   $PD^3$-group at
$p$. Then the following holds.
\begin{enumerate}
\item If  the trivial $G$-module $\F_p$ is self-dual then
$$\ochi_2(G,\F_p)=-\dim H^3(G,\F_p)=-1.$$
In particular, if $G$ is a pro-$p$ group, then $\df_p(G)=0$.
\item Let $M$  be a finite self-dual $\F_p[[G]]$-module, then $\dim H^3(G,M)=\dim M^G$ and
$$\ochi_2(G,M)= -\frac{\dim M^G}{\dim M}.$$
\item Let $N$ be a closed  subgroup of $G$ such that $[G:N]_p$ is infinite. Then $\cd_p(N)\le 2$.
 \end{enumerate}
 \end{prop}
\begin{proof} Let $M$  be a finite  $\F_p[[G]]$-module. By
\cite[Theorem 3.4.6]{NSW}, $\dim H^{n-i}(G,M)=\dim H^i(G, M^*)$.   This implies  the first and second  statements.

 In order to prove (3) we have to note that if $U$ is an open subgroup of $G$, then the correstriction map
 $H^3(U,\F_p)\to H^3(G,\F_p)$ is an isomorphism (see proof of Proposition 30, item (5), in \cite[I.\S 4.5]{Serre}). Hence the restriction map $H^3(G,\F_p)\to H^3(U,\F_p)$ is trivial if $p$ divides $[G:U]$. Therefore if $L$ is an open subgroup of $N$
 $$H^3(L,\F_p)={\displaystyle\lim_{\displaystyle\longrightarrow}}_{L<U\le_o G }H^3(U,
\F_p)=0.$$
because $[G:L]_p$ is infinite.  By Proposition \ref{projdim}, $\cd_p(N)\le 2$.
\end{proof}

In general it is not true that if  $G$ is a   $PD^3$-group at $p$,
then $\df_p(G)=0$, even in orientable case (i.e. $G$ acts
trivially on $I_p(G)$); the  example showing this was obtained in
communication with Peter Symonds.

\begin{example} Let $p\ge 5$ be a prime number and $a\in\Z_p^\vee$ such that $a^2\ne 1$ and $a^{p-1}=1$. Let $x$ be a generator of $C_{p-1}$. Consider $G=(\Z_p^3)\rtimes
C_{p-1}$,  where $x$ acts on $\Z_p^3$ as  multiplication by $a$ on
the first and second coordinates   and as multiplication by
$a^{-2}$   on the third.   So $G$ is an orientable
$PD^3$-profinite group at $p$, because $x$ acts on $I_p(G)$ as
multiplication by $(a\cdot a\cdot a^{-2})^{-1}=1$.

Set $H=\Z_p^3$ and let $M=\F_p$ be a $G$-module such that $H$ acts
trivially and $x$ acts as multiplication by $a^{-1}$.  By
Proposition \ref{consLHS}(4), $$H^\bullet(G,M) \cong
H^\bullet(H,M)^{G/H}.$$
 Moreover, $ H^\bullet(H,\F_p) \cong \bigwedge^\bullet~
H^1(H,\F_p)$. Thus $\dim H^1(H,M)=3$ with eigenvalues $a^{-2}$,
$a^{-2}$ and $a$ for the action of  $x$ and  $\dim H^2(H,M)=3$
with eigenvalues $1$, $1$ and $a^{-3}$ for the action of  $x$.
Therefore,
$$
\dim H^1(G,M) = \dim (H^1(H,M) ^{G/H}=0 $$ and $$ \dim H^2(G,M) =
\dim (H^2(H,  M)^{G/H}\ge 2.
$$
Thus,$$ \ochi_2(G,M)=\frac { - \dim  H^2(G,M) + \dim H^1(G,M) -
\dim H^0(G,M) } { \dim(M) } \le -2
$$ and so $\df_p(G)\leq 1+ \ochi_2(G,M) \le -1$.

 \end{example}

The following theorem is a consequence of   \cite[Theorem 3.7.4]{NSW}.
\begin{theorem}\label{duality} Let $1\to N\to G\to G/N\to 1$ be an exact sequence of profinite groups such that

a) $G$, $N$ and $G/N$ are $p$-$FP_\infty$,

b) $\cd_p(G/N)<\infty$.

Then if two of three groups are $PD_{\bullet}$-groups at $p$, so
is the third. Moreover, $\cd_p(G)=\cd_p(G/N)+\cd_p(N)$.
\end{theorem}
The next result is an analog of Theorem \ref{def1}.

\begin{theorem}\label{def0}
Let $G$ be a finitely generated profinite $PD^3$-group  at $p$,
$K\le N$ two normal subgroup such that $|G/N|_p$ is infinite and
$M$ is a finite self-dual $\F_p[[G]]$-module on which $K$ acts
trivially. Suppose that
$$
\inf \left\{\ \frac{d_{p}(H,K)-1 }{[G:H]_p}\  \vrule \ N<H\le_0 G\
\right\}=0.
$$
Then, $\dim H^2(N,M)\le \dim M^N $.

Moreover,  either $\cd_p(N)\le 1$ or a $p$-Sylow subgroup of $G/N$
is virtually cyclic.
\end{theorem}

\begin{proof}
 First, let us show that $\dim H^2(N,M)\le   \dim M^N$.  By way of contradiction let us assume that $\dim H^2(N,M)>
\dim M^N$. Let $G=H_1>H_2>\ldots$ be a chain of open normal
subgroups such that $\cap_i H_i=N$. Thus we have that
$$N=\displaystyle \underleftarrow{\lim}\  H_i.$$
Hence $H^2(N,M)=\underrightarrow{\lim}\ H^2(H_i,M)$. Since $\dim
H^2(N,M)>\dim M^N$, we obtain that there exists $j$ such that if
$\beta$ denotes  the restriction map $ H^2(H_j,M)\to H^2(N,M)$,
then $\dim \im \beta\ge \dim M^N +1$.

Let $N\le H$ be a open normal subgroup of $G$ contained in $H_j$
and $P$ a subgroup of $H_j$ containing $H$ such that $P/H$ is a
$p$-Sylow subgroup of $H_j/H$. Since $P$ is also a $PD^3$-group at
$p$ and $M$ is self-dual,
 Proposition \ref{proppd3}  gives that
$$\ochi_2(P,M)\ge  -\frac{\dim M^N}{\dim M}.$$ Therefore by Proposition \ref{restr},
$$
\ochi_1(H,M)\ge    ( \ochi_2(P,M)+\frac{\dim(\im
\beta)}{\dim(M)})[P:H]\ge \frac{[P:H]}{
\dim(M)}=\frac{[G:H]_p}{[G:H_j]_p\dim(M)}.
$$

 Hence $d_{p}(H,K)\ge 1+ \frac{[G:H]_p}{[G:H_j]_p\dim(M)}$. Now, if $H$ is an
arbitrary open subgroup containing $N$, applying Lemma \ref{cof}
for $H$ and $H \cap H_j$ and the latter inequality for $H\cap
H_j$, we get
$$\frac{d_{p}(H,K)-1}{[G:H]_p} \ge \frac{d_{p}((H\cap H_j),K )-1}{[H:(H\cap H_j )][G:H]_p}\ge
 \frac{1}{[G:H_j]_p[H:(H\cap H_j )]_{p^\prime}\dim(M)}.$$
 \medskip
  Since $H_j$ is fixed, $\frac{1}{ [H:(H\cap H_j )]_{p^\prime}}$ has positive lower bound that gives a contradiction   with the hypothesis.  Hence, $\dim H^2(N,M)\le \dim M^N$.

Note that by Proposition \ref{proppd3}, $\cd_p(N)\le 2$. Let us
now analyze the case $\cd_p(N)=2$. In this case there exists an
open subgroup $L$ of $N$ such that $H^2(L,\F_p)\ne 0$. Since
$L=U\cap N$ for some open subgroup $U$ of $G$, we conclude that
$H^2(L,\F_p)=\F_p$. Now, Theorem \ref{comptes} implies that $U/L$
and so $G/N$ are of virtual cohomological $p$-dimension 1.
Without loss of generality let us assume that $\cd_p(U/L)=1$.

 If a $p$-Sylow subgroup of $U/L$ is not  cyclic then, by Lemma
\ref{propSylow}, there exists an open subgroup $V$ of $U$
containing $L$ such that $\ochi_1(V/L,\F_p)\ge 1$ and $V$ acts
trivially on $H^2(L,\F_p)$. Thus, by Corollary \ref{consLHS}(4),
$$
1=\dim  H^3(V,\F_p)=\dim H^1(V/L,H^2(L,\F_p))>1,
$$
a contradiction. Hence, the Sylow pro-$p$ subgroup of $U/L$ is
cyclic. \qed \end{proof}

Now we need the following.

\begin{lemma}\label{pd1} Let $G$ be a profinite group with infinite cyclic  Sylow pro-$p$ subgroup. Then there  is an open subgroup $U$ isomorphic to a  semidirect product of a profinite pro-$p'$ group and $\Z_p$.
In particular, $U$ and therefore $G$ are profinite $PD^1$-groups
at $p$.
\end{lemma}
\begin{proof}  We can find an open subgroup $U$ such that $H^1(U,\F_p)= \F_p$. Then Lemma \ref{propquot} implies that $U_{[p]}\cong \Z_p$. Therefore, the kernel $N$ of the natural map $U\to U_{[p]}$ is a pro-$p'$ group.

\end{proof}
Now we are ready to prove Theorem \ref{thpd3}.
\begin{theorem}
Let $G$ be a profinite $PD^3$-group at a prime $p$ and $N$ be a
finitely generated normal  subgroup of $G$ such that $|G/N|_p$ is
infinite and $p$ divides $|N|$. Then either $N$ is $PD^1$ at $p$
and $G/N$ is virtually $PD^2$ at $p$ or $N$ is $PD^2$ at $p$ and
$G/N$ is virtually  $PD^1$ at $p$.
\end{theorem}

\begin{proof}
During this proof when we write $PD^n$ we shall mean $PD^n$ at
$p$.

\

{\it Claim 1.} Assume that the $p$-Sylow subgroup of $N$ is not
cyclic. Then

\begin{enumerate}

\item[(i)] The Sylow pro-$p$ subgroup of $G/N$ is virtually
cyclic; \item[(ii)] $N$ is of type $p$-$FP_\infty$; \item[(iii)]
$G/N$ is virtually $PD^1$ at $p$ and   $N$ is $PD^2$ at $p$.
\end{enumerate}

\

(i) By Lemma \ref{propSylow} we can find an open subgroup $J$ of
$G$ such that $\ochi_1(J\cap N,\F_p)\geq 1$. Applying Proposition
\ref{fgns} and repeating the same argument as in (\ref{t-U/NJ}),
we obtain that there exists an open subgroup $V$ of $J$ containing
$J\cap N$ such that for any open subgroup $U$ of
 $J$ containing $J\cap N$, $\ochi_1(U/N\cap J,\F_p)\le 0$.
 This means that the Sylow pro-$p$ subgroup of $G/N$ is virtually
 cyclic.

 \

 For simplicity let us assume that the Sylow pro-$p$
 subgroup of $G/N$ is cyclic, and so $\cd_p(G/N)=1$.

\medskip

 (ii)  Note that if $U$ is an open subgroup of $G$ containing $N$ and $M$
a finite $\Z_p[[U]]$-module, then by Corollary \ref{consLHS}(1),
$$\dim H^1(U,M) \le \dim H^1(U/N, M^N)+\dim H^1(N,M)^{U/N}.$$ Since by Lemma \ref{pd1}, $U/N$ is $PD^1$ at $p$ one has
$$\dim H^1(U/N, M^N)=\dim H^0(U/N, (M^N)^\vee)=\dim ((M^N)^\vee)^{U/N}$$ and using $$ \dim
((M^N)^\vee)^{U/N}+\dim H^1(N,M)^{U/N}\leq \dim M + \dim
H^1(N,M)$$ one has $$\dim H^1(U,M)\leq \dim M + \dim H^1(N,M).$$
Thus, $d_p(U)\le 1+d_p(N)$ and so we may apply Theorem \ref{def0}
for $K=1$.

 Let $S$ be an irreducible finite $\F_p[[N]]$-module.  There exists an open subgroup $V$ of $G$ such that $V\cap N$ acts trivially on $S$. We convert $S$ in a $VN$-module by assuming that the elements of $V$ act trivially on $S$. Since $VN$ is open subgroup of  $G$, it is also $PD^3$ at $p$. By  Theorem \ref{def0}, $\dim H^2(N,  S\oplus  S^\vee)\le 2\dim S$. Hence $ \dim H^2(N,S)\le 2\dim S$.

 Therefore since, by Proposition \ref{proppd3}  ,  $N$ has cohomological
$p$-dimension 2, it follows from Theorem \ref{critfp} that  $N$ is
of type $p$-$FP_\infty$.

 \medskip

 (iii) By Lemma \ref{pd1}, $G/N$ is $PD^1$ and we can apply Theorem
\ref{duality} and conclude that $N$ is $PD^2$.

\

 \

 {\it Claim 2.} Suppose that the Sylow pro-$p$ subgroup of $N$ is  cyclic.  Then $N$ is $PD^1$ at $p$ and $G/N$ is virtually $PD^2$ at $p$

 \

   Applying Corollary \ref {corwz}, we obtain that $\vcd_p(G/N)=2$. Let $U$ be an open subgroup of $G$ containing
   $N$ such that $\cd_p(U/N)=2$. Since $N$ is finitely generated,
$U/N $ is of type $p$-$FP_2$ by Corollary \ref{pf2} and so it is
also of type $p$-$FP_\infty$.
  By Lemma \ref{pd1}, $N$ is $PD^1$. Now, Theorem \ref{duality} implies that $U/N$ is $PD^2$ and so $G/N$  is virtually $PD^2$.
\end{proof}

\begin{corollary} Let $G$ be a finitely generated pro-$p$ $PD^3$-group  at
$p$ and $ N$ a normal subgroup of infinite index. Suppose that $G$ has subexponential subgroup growth.
 Then $N$ is either free pro-$p$ or a  Demushkin group.
\end{corollary}

\begin{proof} Proposition \ref{subexp} implies that $$\inf \left\{\ \frac{d(H)-1 }{[G:H]}\  \vrule \ N<H\le_0 G\ \right\}=0.$$ Therefore, by Theorem \ref{def0}, $\dim H^2(N,\F_p)\le 1$. If $H^2(N,\F_p)=0$ then $N$ is free pro-$p$. If
$H^2(N,\F_p)=\F_p$ then Theorem \ref{def0} implies that $G/N$ is
virtually cyclic. Hence since $G/\Phi(N)$ is not of exponential subgroup growth,
$N/\Phi(N)$ should be  finite. Thus, $N$ is finitely generated. By the
previous theorem, $N$ is a Demushkin group.
\end{proof}

\section{Applications for discrete groups}

In this section we shall describe applications of our profinite
results to discrete finitely generated groups.

\subsection{Good groups}\label{Good groups}

Let $\Gamma$ be a group,
$\widehat \Gamma$ its profinite completion. The group $\Gamma$ is
called {\bf $p$-good} if the homomorphism of cohomology groups
$$i^n(M): H^n(\widehat \Gamma,M)\longrightarrow H^n( \Gamma,M)$$
induced by the natural homomorphism $i:\Gamma\longrightarrow
\widehat \Gamma$    is an isomorphism for every finite $p$-primary
$\Z[\Gamma]$-module $M$ and for all $n\ge 0$. The group $\Gamma$ is
called {\bf good} if it is $p$-good for all primes $p$. This notion was introduced by Serre (see \cite[I.2.6]{Serre}) and has been studied recently in several papers (see, for example, \cite{GJZ}).

Let $\Gamma_{\hat p}$ be the pro-$p$ completion of $\Gamma$, and let $i_p:\G\to \G_{\hat p}$ denote the canonical map. This map induces natural homomorphisms
$$
i_p^n(M): H^n(\Gamma_{\hat p},M)\longrightarrow H^n( \Gamma,M)
$$ for all $n\ge 0$ and for all finite $\Z[\Gamma]$-modules $M$ of $p$-power order for which all composition factors are trivial $\Gamma$-modules.
The group $\Gamma$ is called {\bf pro-$p$ good}  if $i_p(\F_p)$ is an isomorphism for all $n$.
Note that this also implies that $i_p^n(M)$ is an isomorphism for all $n\ge 0$ and for all finite $\Z[\Gamma]$-modules $M$ of $p$-power order for which all composition factors are trivial $\Gamma$-modules.

The following relation between $p$-goodness and pro-$p$ goodness
was discovered by  Thomas  Weigel \cite{We}:
\begin{prop}\label{p-good-pro} Let $\Gamma$ be a group. Assume that all the subgroups of $\Gamma$ of finite index are pro-$p$ good. Then $\Gamma$ is $p$-good.
\end{prop}
The following result is a variation of a result of  Serre from \cite[I.2.6]{Serre}.
\begin{prop}\label{prserre}
Let $\Gamma$ be a discrete group, $G$ a profinite group and $\phi:
\Gamma\to G$ a homomorphism with dense image. For any finite
(topological) $G$-module $M$, denote by $\phi^{n}(M)$ the
restriction  map $\phi^{n}(M): H^n(G,M)\to H^n(\G, M)$. Then the
following properties are equivalent:
\begin{itemize}
\item[$A_k$]  For every finite $\F_p[[G]]$-module $M$, $\phi^n(M)$ is bijective for all $n\le k$ and injective for $n=k+1$.
\item [$B_k$] For every finite $\F_p[[G]]$-module $M$, $\phi^n(M)$ is surjective for all $n\le k$.
\item[$D_k$] ${\displaystyle\lim_{\displaystyle\longrightarrow}} _{U\le_o G} H^n(\Gamma \cap U, \F_p)=0$
 for all $n\le k$.

\end{itemize}

 \end{prop}

As an application of some results of Section \ref{Normal subgroups of profinite groups of positive deficiency} we obtain the following theorem.
\begin{theorem}\label{gooddef1} Let $\Gamma$ be finitely presented group of deficiency 1 and cohomological dimension 2. Assume that $\ochi_2(U,\F_p)\le 0$ for any open subgroup $U$ of $\Gamma_{\hat p}$. Then $\Gamma$ is pro-$p$ good.
\end{theorem}

\begin{proof}
 First observe that the condition $B_1$ from Proposition \ref{prserre} always holds in the case $G=\G_{\hat p}$.
   Hence  $A_1$ and $D_1$ holds. In particular,  for any $U\le _o \Gamma_{\hat p}$,
$$\begin{array}{lll}\dim H^1(U,\F_p)& = &\dim H^1(\G_{\hat p}, Coind_U^{\G_{\hat p}}(\F_p))\\ & &\\ &=& \dim H^1(\G, Coind_{U\cap \G}^{\G}(\F_p))= \dim H^1(\Gamma\cap U,\F_p)\\ &&\\\dim H^2(U,\F_p)& = &\dim H^2(\G_{\hat p}, Coind_U^{\G_{\hat p}}(\F_p))\\ & &\\ &\le& \dim H^2(\G, Coind_{U\cap \G}^{\G}(\F_p))= \dim H^2(\Gamma\cap U,\F_p).\end{array}$$
Thus, since any subgroup of $\Gamma$ of finite index has positive deficiency, we obtain that
$$\begin{array}{lll}
1 & \le & \df (\Gamma\cap U)\le \dim H^1(\Gamma\cap U,\F_p)-\dim H^2(\Gamma\cap U,\F_p)\\ &&\\ &\le& \dim H^1( U,\F_p)-\dim H^2(  U,\F_p)= \ochi_2(U,\F_p)+1\le 1.\end{array}$$
Therefore $  \dim H^2(\Gamma\cap U,\F_p)=  \dim H^2(  U,\F_p)$ and so
$${\displaystyle\lim_{\displaystyle\longrightarrow}} _{U\le_o G} H^2(\Gamma \cap U, \F_p)={\displaystyle\lim_{\displaystyle\longrightarrow}} _{U\le_o G} H^2( U, \F_p)=H^2(\{1\},\F_p)=0.$$
Hence the condition $D_2$ from Proposition \ref{prserre} holds. Thus, by the condition $A_2$,  $i_p^k(\F_p)$ is an isomorphism for $k\le 2$ and $H^3(\G_{\hat p},\F_p)\le H^3(\Gamma,\F_p)=0$.  Hence $\cd_p(\G_{\hat p})=2$ by Proposition \ref{projdim}.
 We conclude that $\Gamma$ is pro-$p$ good.
\end{proof}

\begin{corollary}\label{semi}
Let $\Gamma$ be finitely presented group of deficiency 1 and
cohomological dimension 2. Assume that $\widehat \Gamma$ is a
semidirect product of  a finitely generated  profinite group $P$
and $\hat \Z$. Then $\Gamma$ is good and pro-$p$ good for any
prime $p$.
\end{corollary}

\begin{proof}
Let $H$ be a subgroup of $\Gamma$ of finite index and $p$ a prime.
Note that $\cd(H)=2$ and its deficiency is positive. Since
$\widehat H$  is a subgroup of $\widehat \G$, it is also a
semidirect product of a finitely generated profinite group and
$\hat \Z$. Hence $H_{\hat p}\cong \widehat H_{[p]}$ is a
semidirect product of a  finitely generated pro-$p$ group and
$\Z_p$. Thus, by Corollary \ref{def2ns}, if $U$ is an open
subgroup of $H_{\hat p}$, then  $$\ochi_2(U,\F_p)=\df_p(U)-1\le
0.$$ Applying, Theorem \ref{gooddef1}, we obtain that $H$ is
pro-$p$ good. Proposition  \ref{p-good-pro} implies that $\Gamma$
is good.
\end{proof}

Let $\Gamma $ be a finitely presented group. A {\bf chain} in $\Gamma $ is
a decreasing infinite sequence $\Gamma =\Gamma _{0}>\Gamma _{1}>\ldots $ of
subgroups of finite index in $\Gamma $. The chain is {\bf normal} if all $
\Gamma _{n}$ are normal in $\Gamma $. From a result of Luck (see, for example, \cite{Luek}) we know that for any normal chain $\{\G_i\}$ with trivial intersection there exists
$$\lim_{i\to \infty} \frac{\dim H^1(\Gamma_i,\Q)}{[\Gamma:\Gamma_i]}.$$
Moreover this limit does not depend on $\{\G_i\}$ and if $\Gamma$ is infinite it is equal to the firtst $L^2$- Betti number $\beta_1^{(2)}(\Gamma)$.  We need the  following auxiliary proposition. We say that a group $\Gamma$ is {\bf residually-$p$} if $G$ is residually (a finite $p$-group) group

\begin{prop} Let $\Gamma$ be residually-$p$  group. Then

a)  $\beta_1 (\Gamma)=\dim_{\Q} H^1(\Gamma, \Q)\ge \df_p(\G_{\hat p})$,

b) $\beta^{(2)}_1(\Gamma)\ge \ochi_2(\G_{\hat p},\F_p)$.
\end{prop}

\begin{proof}
In order to prove a), note that $ \df_p(\G_{\hat p})$ is the
difference between the minimal number of generators and the
minimal number of relations of $\G_{\hat p}$ as a  pro-$p$ group.
Hence arguing as in \cite[Window 5, Sec.1, Lemma 3]{LS}, we obtain
that $$\dim_{\Q_p} \Q_p\otimes_{\Z_p}
 \G_{\hat p}/[\G_{\hat p}, \G_{\hat p}]
 \ge \df_p(\G_{\hat p}).$$
Since $\dim_{\Q} H^1(\Gamma, \Q)=\dim_{\Q_p}   \Q_p\otimes_{\Z_p}
 \G_{\hat p}/[\G_{\hat p}, \G_{\hat p}]$ we obtain a).

Let $U$ be  open subgroup of $\G_{\hat p}$. By Lemma \ref{cof} and
a),
$$
\frac{\dim_\Q H^1(U\cap \Gamma,\Q)-1}{[\G:(U\cap \G)]} \ge \frac{
\df_p(U)-1}{[\G_{\hat p}:U]}\ge \df_p(\Gamma_{\hat
p})-1=\ochi_2(\Gamma_{\hat p},\F_p).
$$
Since $\Gamma$ is residually-$p$ group, then if $\{U_i\}$ is a
normal chain of $\G_{\hat p}$ with trivial intersection, then
$\{\G\cap U_i\}$ is a normal chain with trivial intersection of
$\Gamma$. Hence $\beta^{(2)}_1(\Gamma)\ge \ochi_2(\Gamma_{\hat
p},\F_p).$ \qed \end{proof}
Now we are ready to prove another criterion of goodness.
\begin{theorem}\label{betti0}
Let $\G$ be finitely presented group. Suppose that $\G$ is
virtually residually-$p$ group and $\beta_1^{(2)}(\Gamma)=0$.
Assume also that either

a) $\cd (\G)=2$ and it is of deficiency 1 or

b) $\G$ is an orientable Poincare duality group of dimension 3.

 Then $\Gamma$ is  $p$-good. Moreover, if $G$ is residually-$p$
 then $G$ is pro-$p$ good.
 \end{theorem}
 \begin{proof} Assume first that $\Gamma$ is residually-$p$ group. Let $U$ be an open subgroup of  $G_{\hat p}$. By the previous proposition  and \cite[Theorem 1.35 (9)]{Luek}, $$\df_p(U)-1=\ochi_2(U,\F_p)\le \beta_1^{(2)}(\G\cap  U)=[\G:(\G\cap U)]\beta_1^{(2)}(\Gamma)=0.$$

 If we are in the case a), then by Theorem \ref{gooddef1}, $\Gamma$ is pro-$p$ good and the same argument shows that any subgroup of $\Gamma$ of finite
 index is pro-$p$ good. Hence by Proposition \ref{p-good-pro}, $\Gamma$ is also
 $p$-good.

Assume now that we are in the case b). Then the same conclusion
follows from \cite[Theorem A and B]{KZ}.

 Now if $\Gamma$ is not residually-$p$, there is a subgroup of
 finite index of $\Gamma$ which is and since $p$-goodeness is
 preserved by overgroups of finite index (cf. Proposition
 \ref{prserre}) $\Gamma$ is $p$-good.

\qed \end{proof}
\subsection{Ascending HNN-extensions}\label{Ascending HNN-extensions}

In this section we study mapping tori of injective endomorphisms
of free groups. For a free group $F_n:=\langle x_1,\ldots,
x_n\rangle$ ($n\in\N$) let $\phi : F_n\to F_n$ be an endomorphism.
The HNN-extension
\begin{equation}
M_\phi:=\langle \, x_1,\ldots, x_n, t \ \vrule \ t^{-1} x_i
t=\phi(x_i) \ {\rm for}\ i=1,\ldots n \ \rangle
\end{equation}
is traditionally called the mapping torus of $\phi$. Sometimes we
shall also say that $M_\phi$ is the ascending HNN-extension of
$\phi$ or of the free group $F_n$. Groups of this type often
appear in group theory and topology and were extensively studied
see \cite{FH}, \cite{BS}, \cite{IK} for example. In particular,
many one-relator groups are ascending HNN-extensions of free
groups and  many of such  groups are hyperbolic. Our methods are
applicable to the study of mapping tori because
\begin{itemize}
\item the group $M_\phi$ has positive deficiency, \item the
cohomological dimension of $M_\phi$ is less or equal to $2$.
\end{itemize}
The second property follows by an application of the
Mayer-Vietoris of cohomology while the first property is obvious.

Let us discuss a simple example. Let the endomorphism $\phi_1 :
F_1=\langle x\rangle\to F_1$ be given by $\phi_1(x):=x^2$. It is
elementary to see that the corresponding mapping torus is
metabelian and in fact isomorphic to a split extension
$$M_{\phi_1}=\langle x,t\mid
x^t=x^2\rangle\cong \Z\left[ \frac{1}{2}\right]\rtimes \Z.$$ Where
the generator $1$ of the infinite cyclic group acts by
multiplication by $2$ on the group $\Z[1/2]$ of rational numbers
with a $2$-power denominator. The commutator subgroup of
$M_{\phi_1}$ is equal to $\Z[1/2]$ appropriately embedded in
$M_{\phi_1}$. From here it is easy to detect the profinite
completion of
 $M_{\phi_1}$.
\begin{prop}\label{example}
We have
$$\widehat{M_{\phi_1}}= \widehat{\Z[1/2]}\rtimes \widehat{\Z}.$$
The profinite completion $\widehat{\Z[1/2]}$ is a projective but
not a free profinite group.
\end{prop}
In fact the isomorphism
$$\widehat{\Z[1/2]}\cong \prod_{p\ne 2} \Z_p$$
shows that all $p$-Sylow subgroups of $\widehat{\Z[1/2]}$ are free
pro-$p$ groups which implies that $\widehat{\Z[1/2]}$ is
projective. Our example also shows that the projectivity of
$[G,G]$ in Corollary \ref{cora} can not be replaced by freeness.

Corollary \ref{profinite} allows to establish the structure of the
profinite completion of this important class of groups.

\begin{theorem}\label{HNN}
Let $M_\phi$ be an ascending HNN-extension of a finitely generated
free group $F=F_n$ of rank $n\in\N$ with respect to an
endomorphism $\phi: F\to F$. Let $P$ be the closure of $F$ in the
profinite completion $\widehat M_\phi$. Then $P$ is normal in
$\widehat M_\phi$, the profinite completion of $M_\phi$ is
isomorphic to the split extension $\widehat M_\phi=P\rtimes
\widehat{\mathbb{Z}}$ and $P$ is a projective profinite group. The
group $P$ is free profinite of rank $n$ if and only if $\phi$ is
an automorphism.
\end{theorem}

\begin{proof}
Let $P$ be  the closure of the image of $F$ in $\widehat M_\phi$.
To see that $P$ is normal consider any finite quotient $H$ of
$M_\phi$. Then observe  that the images of $F$ and $F^{t}$ in $H$
coincide because in a finite group conjugate subgroups have to be
of the same order. Thus $\widehat M_\phi=P\rtimes
\widehat{\mathbb{Z}}$.

The group $M_\phi$ has positive deficiency  and therefore so has
its profinite completion  $\widehat M_\phi$. Corollary
\ref{profinite} implies that $P$ is projective.

If $\phi$ is an automorphism clearly $P=\widehat F$. Suppose now
that $\phi$ is not an isomorphism. Let $F_0$ be the image of
$\phi$ it is a finitely generated subgroup of $F$ distinct from
$F$. By a result of M. Hall $F_0$ is not dense in the profinite
completion of $F$, see \cite{LyS}, Proposition 3.10. To see that
$P$ is not free profinite of rank $n$, it suffices to show that
the profinite topology of $M_\phi$ does not induce the full
profinite topology on $F$. But this follows from the theorem of
Hall since as was just observed $F$ coincides with $F_0$ in every
finite image of $M_\phi$.\end{proof}

\begin{remark}\label{projectivity} \rm If $f$ is not an automorphism the
theorem in principal
allows that $P$ is a free profinite group of rank
less than $n$. It is, however, easy to give criteria when this does
not happen. Indeed, if $F$ admits a finite quotient $F/N$ modulo a
characteristic subgroup $N$ such that $d(F/N)=n$ and $FN=f(F)N$
then $d(P)=n$. So for instance, if $f(x_i)=x_i^p$ for $i=1,\ldots
n$ for some prime $p$ then $F/N$ can be taken to be elementary
$q$-group of rank $n$ where $q$ is coprime to $p$.\end{remark}

\begin{problem} Describe the projective groups $P$ obtained in this
manner. \end{problem}

The following interesting example  was communicated to us by I. Kapovich.

\begin{example}
Let $F:=F_3=\langle \, a,\, b,\, c\,\rangle$ be a free group of
rank $3$. Let the endomorphism $\phi:F\to F$ be given by
$$\phi(a)=a,\quad \phi(b)=a^{-1}ca, \quad \phi(c)=a^{-1}bab^{-1}.$$
Let
$$\G:=M_\phi=\langle a,b,c,t\ \vrule\  tat^{-1}=a,
tbt^{-1}=a^{-1}ca, tct^{-1}=a^{-1}bab^{-1}\rangle$$ be the corresponding
ascending HNN-extension.
Then $\phi(F)=\langle a,c, bab^{-1}\rangle$  is a proper subgroup of F, so that
$\G$ is a strictly ascending HNN-extension.

However, we can also rewrite the defining relations for $G$ as
follows:
$$a^{-1}ta=t,\quad
a^{-1}ca=tbt^{-1},\quad a^{-1}ba=tct^{-1}b.$$
Thus $G$ is an HNN-extension of $H:=\langle\, t,\, b,\, c\, \rangle$
with respect to
$\psi: H\to H$ where
$$\psi(t)=t, \quad \psi(c)=tbt^{-1} , \quad \psi(b)=tct^{-1}b$$
and with stable letter $a$.
We have $\psi(H)=\langle \, t,tbt^{-1}, tct^{-1}b\,
\rangle=\langle t,b,c \rangle=H$. Thus $\psi$ is an automorphism
of $H$ and hence $H$ is normal in $G$.
\end{example}

This example combined with Remark \ref{projectivity} shows that
the profinite completion $\widehat \G$ can be written as semidirect
product $P\rtimes \widehat \Z$ of a projective (non-free) finitely
generated profinite group and as a semidirect product $F\rtimes
\widehat\Z$ of a free profinite group of rank 3 and $\widehat\Z$.


The next result that comes as an application of Corollary
\ref{semi} is that   an ascending HNN-extension  is good.
\begin{theorem}\label{good} An ascending HNN-extension $\G$ of a
finitely generated free group is good and pro-$p$ good for every prime $p$.\end{theorem}

\begin{proof}
By Theorem \ref{HNN}, $\widehat{\G}\cong P\rtimes \widehat{\Z}$
where $P$ is finitely generated. Now  by Corollary
\ref{semi}, $\G$ is good and pro-$p$ good for every prime $p$. \end{proof}
\begin{remark} \rm We note that Lorensen
\cite{L} proved that if $M$ is a finite $\Z[\Gamma]$-module, then
$$
i^n(M): H^n(\widehat{\G} ,M)\longrightarrow H^n( \G,M)$$ is an isomorphism
for $n\le 2$. Thus, the goodness of $\G$ also follows directly from his result and our Theorem \ref{HNN}, because $\cd(\Gamma)=\cd(\widehat \Gamma)=2$. \end{remark}

We finish the subsection with a similar construction in the
profinite setting that not necessarily arises from the profinite
completion of the respective abstract construction, but gives  a
similar result. The proof is analogous to the above.

\begin{prop} Let $F(x_1,\ldots x_n)$ be a free profinite
group of finite rank $n$ and $f:F\longrightarrow F$ an
endomorphism. Form a profinite
 HNN-extension $G=\langle F,t\mid x_i^t=f(x_i)\rangle$.
Then  $G = P\rtimes \widehat{\mathbb{Z}}$, where $P$ is
projective. Here $P$ is the image of the natural homomorphism
$F\to G$.
\end{prop}

\subsection{More examples}

This section contains some more constructions for groups for which
our techniques are applicable.

Let $F=F(x_1,\ldots x_n, y_1\ldots y_m)$ be a free group of finite
rank $n+m$. Let
$$f_1:F(x_1\ldots x_n)\longrightarrow F(x_1\ldots
x_n), \quad f_2:F(y_1\ldots y_m)\longrightarrow F(y_1\ldots y_m)$$
 be  injective endomorphisms.

\begin{theorem}  Let
$G=\langle F,t\mid x_i^t=f_1(x_i),\ y_j^{t^{-1}}=f_2(y_j) \rangle$. Then
the profinite completion of $G$ is $\widehat G=P\rtimes
\widehat{\mathbb{Z}}$, where $P$ is projective. $P$ is free
profinite of rank $n+m$ if and only if $f_1,\, f_2$ are automorphisms.
\end{theorem}

\begin{proof} The group $G$ has positive
deficiency  and therefore so is $\widehat G$.
Let $P$ be  the closure of the image of $F$ in $\widehat G$. To
see that $P$ is normal fix any finite quotient $\bar G$ of $G$ and
use bar to denote  the images in $\bar G$. Then one  observes
that $\bar F(x_1\ldots x_n)=\bar F(x_1\ldots x_n)^{\bar t}$ and
$\bar F(y_1\ldots y_m)=\bar F(y_1\ldots y_m)^{\bar t}$ because
finite conjugate groups have to be of the same order. Thus
$\widehat G=P\rtimes \widehat{\mathbb{Z}}$ and so by Corollary
\ref{profinite} $P$ is projective. \end{proof}

\begin{theorem}
Let $G=\langle F,t\mid x_i^t=f_1(x_i),\ y_j^{t^{-1}}=f_2(y_j)
\rangle$. Then $G$ is residually finite and good.
\end{theorem}

\begin{proof}

Put $F_1=F(x_1\ldots, x_n)$ and $F_2=F(y_1\ldots y_m)$. The group
$G$ can be represented as an amalgamated free product of ascending
HNN-extensions $G_1*_Z G_2$, where $G_1=\langle F_1,t\mid
x_i^t=f_1(x_i)\rangle$, $G_2=\langle F_2,t\mid
y_j^{t^{-1}}=f_2(y_j)\rangle$ and $Z=\langle t\rangle$. Then by
Exercise 9.2.7 in \cite{RZ} $G$ is residually finite and by
Proposition 3.5 in \cite{GJZ} combined with Theorem \ref{good} is
good.
\end{proof}

\subsection{Implications for the congruence kernel}

Here we give some results relating to the structure of the
congruence kernels of certain arithmetic groups.

Recall that a lattice in $\SL_2(\CC)$ is a discrete group
(Kleinian group) of finite covolume. One particular family of
lattices is a family of arithmetic groups. We recall the
definition of an arithmetic group in this case; see \cite{EGM},
\cite{MaRe} for more details. Let $k$ be a number field with
exactly one pair of complex places and let $A$ be a quaternion
algebra over $k$ which is ramified at all real places. Let $\rho$
be a $k$-embedding of $A$ into the algebra ${\rm M}_2(\CC)$ of two
by two matrices over $\CC$ (using one of the complex places). Let
$\cal O$ be the ring of integers of $k$ and let $\cal R$ be a
$\cal O$-order of $A$. Let $A^1({\cal R})$ be the corresponding
group of elements of norm one. It is well known that
$\rho(A^1({\cal R}))$ is a lattice in $\SL_2(\CC)$. Then a
subgroup $\Gamma$ of $\SL_2(\CC)$ is an arithmetic Kleinian group
if it is commensurable with some such a $\rho(A^1({\cal R}))$
(groups are commensurable if they have $\SL_2(\CC)$-conjugate
subgroups of finite index). The quotient $\SL_2(\CC)/
\rho(A^1({\cal R}))$ is not compact if $k$ is an imaginary
quadratic number field and if $A={\rm M}_2$.

To define the congruence kernels let us (without loss of
generality) in the following consider the arithmetic Kleinian
group $\Gamma=\rho(A^1({\cal R}))$. The congruence
 kernel ${\bf C}(A,{\cal R})$
is the kernel of the canonical map from the profinite completion
$\widehat {\Gamma} $ of $\Gamma$ to $\rho(A^1(\widehat{\cal R}))$.
Here $\widehat{\cal R}$ stands for the profinite completion of the
ring $\cal R$. The congruence subgroup  problem (in general, for
arithmetic groups) asks whether the congruence kernel is trivial.
If the congruence kernel is finite, i.e. the congruence subgroup
problem has almost positive solution, one says that $\Gamma$ has a
congruence subgroup property. It is  proved by Lubotzky \cite{Lu2}
that the congruence kernels ${\bf C}(A,{\cal R})$ of the
arithmetic lattices in $\SL_2(\CC)$ are infinite. But for some of
the arithmetic Kleinian groups, like for example for the
$\SL_2({\cal O})$ ($\cal O$ the ring of integers in some imaginary
quadratic number field), some further information has been
obtained. For these arithmetic Kleinian groups there are subgroups
of finite index which map onto non-abelian free groups. For the
$\SL_2({\cal O})$ this was proved in \cite{GS}, many more cases
are treated in \cite{Lu3}. The fact that $\Gamma$ has a subgroup
of finite index which maps onto non-abelian free group, leads to
an embedding of the free profinite group on countably many
generators $\widehat F_\omega$ into the corresponding congruence
kernel (see \cite{Lu4}).

Led by a result of Melnikov \cite{Mel} in the case
$\Gamma=\SL_2(\mathbb{Z})$ we ask:
\begin{question}\label{melni}
Is the congruence kernel of an arithmetic Kleinian group
isomorphic to $\widehat F_\omega$? Or more vaguely, what can be
said about the congruence kernel in this case?
\end{question}
Of course the answer to question \ref{melni} is negative if the
cohomological dimension of the congruence kernel is not one. We
shall describe in the following an interesting connection between
question \ref{melni} and certain cohomological problems.

In \cite{LoLu} it is proved that if $\Gamma$ is a lattice in
$\SL_2(\CC)$, then for any chain of normal subgroups $\Gamma_i$ of
finite index of $\Gamma$
  with trivial intersection the numbers
$$\frac{\dim H^1(\Gamma_i,\Q)}{[\Gamma:\Gamma_i]}$$
tend to zero when $i$ tends to infinity (this means that $\beta_1^{(2)}(\Gamma)=0$). Let us formulate the
following analoguous problem for the dimensions of the first
cohomology groups over $\F_p$.

\begin{question}\label{pgradient} Let $\Gamma$ be an arithmetic
Kleinian group and $p$ a prime number. Do the numbers
$$\frac{\dim H^1(\Gamma_i,\F_p)}{[\Gamma:\Gamma_i]}$$
tend to zero when $i$ tends to infinity for any chain of  normal $p$-power index
subgroups $\Gamma_i$ of $\Gamma$  with trivial intersection.
\end{question}
There is an interesting case studied by Calegari, Dunfield
\cite{CaDu1} and Boston, Ellenberg \cite{BE} where the answer to
Question \ref{pgradient} is positive. The paper \cite{CaDu1}
contains the description of a subgroup $\Gamma_{CD}\le \SL_2(\CC)$
with the following properties (proofs in \cite{CaDu1}, \cite{BE}).
\begin{itemize}
\item $\Gamma_{CD}$ is a cocompact arithmetic Kleinian lattice in
$\SL_2(\CC)$,  \item the pro-3 completion of $\Gamma_{CD}$ is
analytic, \item $\dim H^1(\Delta,\F_3)=3$ for any normal subgroup
of 3-power index in $\Gamma_{CD}$.
\end{itemize}
We shall now describe a connection beween the problems posed in
Questions \ref{melni} and \ref{pgradient}. Our methods are
applicable since the discrete subgroups of $\SL_2(\CC)$ have a
very restrictive structure. If $\Gamma$ is a finitely generated
torsion free discrete subgroup of $\SL_2(\CC)$, then the
deficiency of $\Gamma$ is 0 or 1 depending whether the quotient
space $\SL_2(\CC)/\Gamma$ is compact or not. Moreover, if
$\SL_2(\CC)/\Gamma$ is compact, then $\Gamma$ is Poincar\'e
duality groups of dimension 3 and if $\SL_2(\CC)/\Gamma$ is not compact then $\Gamma$ is of cohomological dimension 2. First using Theorem \ref{betti0}, we prove that the arithmetic lattices are good. For the Bianchi groups it was shown in \cite{GJZ}.

\begin{theorem}\label{arithgood} Let $\Gamma$ be an arithmetic Kleinian group in $\SL_2(\CC)$. Then $\Gamma$ is  good and pro-$p$ good for every prime $p$.
\end{theorem}
\begin{proof}
Notice that for any prime $p$ an arithmetic lattice in $\SL_2(\CC)$ is  a virtually residually-$p$ group. Now the result follows from Theorem \ref{betti0}, because $\beta_1^{(2)}(\G)=0$.
\end{proof}
\begin{remark}\rm
 Using a recent deep result of D. Wise \cite{Wi} it is also possible to prove goodness for all fundamental groups of 3-manifolds. We give a sketch of the argument.

Let $\G=\pi_1(M)$ be the fundamental group of a 3-manifold $M$ (possibly with a boundary). Without loss of generality we may assume that the boundary of $M$ is incompressible and $M$ is irreducible, because a free product of good groups is also good. From the  Geometrization Conjecture proved by Perelman, it follows that we can cut $M$ along a finite collection of incompressible tori so that the resulting pieces $\{M_i\}$ are geometric. Wilton and Zalesski \cite{WZ} proved that $\pi_1(M)$ is good if  $\pi_1(M_i)$ are good. Also they proved that $\pi_1(M_i)$ is good if $M_i$ is a Seifert manifold and  it is clear that  $\pi_1(M_i)$ is goodif $M_i$ is a solvmanifold. Thus, we have to show that $\pi_1(M_i)$ is good when $M_i$ is hyperbolic. Hence assume that $M$ is hyperbolic. In the case when $M$ is not virtually Haken the goodness of $\pi_1(M)$ was first observed by Reznikov \cite{Re} (see also \cite{KZ,We}). If $M$ is virtually Haken, then by \cite{Wi} and \cite{Al}, $M$ is virtually fibered over circle and so $\pi_1(M)$ is also good.\end{remark}

Now we are ready to prove our main result.
\begin{theorem}\label{congr} Let $\Gamma$ be an arithmetic Kleinian group in
$\SL_2(\CC)$ and $p$ be a prime number. If the answer to the
Question \ref{pgradient} is positive for all subgroups of $\Gamma$
of finite index, then
   the $p$-cohomological
dimension of the congruence kernel of $\Gamma$ is 1.
\end{theorem}
\begin{proof}  By Theorem \ref{arithgood}, $\G$ is $p$-good and so $G=\widehat \Gamma$ is either a Poincar\'e duality profinite group at $p$
of dimension 3 or has $p$-deficiency 1.

Let $\bf C$ denote the congruence kernel corresponding to
$\Gamma$. Suppose that  the $p$-cohomological dimension of $\bf C$ is not 1.
Then by Proposition \ref{projdim} there exists an open subgroup
$C_0$ of $\bf C$ with $H^2(C_0,\F_p)\ne 0$. Further,
there exists  a subgroup $\Gamma_0$ of finite index in $\Gamma$
which congruence kernel is equal to $C_0$. Hence without loss of
generality we may assume that ${\bf C}=C_0$ and $\Gamma=\Gamma_0$.

Note that
$$H^2({\bf C},\F_p)=\varinjlim_{{\bf C}\le
H\triangleleft_oG}  H^2(H, \F_p).$$  Hence there exist a subgroup ${\bf C}\le
H\triangleleft_o G$ such that the image of the restriction map
$H^2(H,\F_p)\to H^2({\bf C},\F_p)$ is not trivial. Note that
$H\cap \Gamma$ is a congruence subgroup of $\Gamma$ and
$H\cong\widehat{H\cap \Gamma}$. Hence without loss of generality
we may also assume that $H=G$.

Since $\Gamma$ is an arithmetic Kleinian subgroup of $\SL_2(\CC)$
there exists a normal subgroup $N$ of $G$ such that $C\le N$ and
$G/N$ is $p$-adic analytic group, because every finitely generated subgroup of the congrunce completion is $p$-adic analytic (see Window 9 \cite{LS}). Moreover $\Gamma$ is embedded in
$G/N$. Replacing $\Gamma$ by one of its congruence subgroups we
may also assume that  $G/N$ is a pro-$p$ group.

Note that the image of compositions of two restriction maps
$$H^2(G,\F_p)\to H^2(N,\F_p)\to H^2({\bf C},\F_p)$$
is not zero. Hence $H^2(N,\F_p)\ne 0$. Thus, we may apply Theorem
\ref{def1} and Theorem \ref{def0} and obtain that there exists
$c>0$ such that for any open subgroup $N\le H\le_o G$
$$\dim H^1(\Gamma\cap H,\F_p)= d_{p}( H,N)\ge
c[G:H]=c[\Gamma:(\Gamma\cap H)].$$ But this contradicts  the
assumption that the answer to Question \ref{pgradient} is positive
for $\Gamma$.
\end{proof}


\bigskip
{\it Author's Adresses:}

\medskip
Andrei Jaikin-Zapirain\\
Departamento de Matematicas, Facultad de Ciencias,\\
Universidad Autonoma de Madrid,\\
Cantoblanco Ciudad Universitaria,\\
28049 Madrid\\
Spain\\
and:
Instituto de Ciencias Matematicas-CSIC, UAM, UCM,UC3M

\medskip
Aline G.S. Pinto\\
Departamento de Matem\'atica,\\
~Universidade de Bras\'\i lia,\\
70910-900 Bras\'\i lia DF,\\
Brazil

\medskip
Pavel A. Zalesski\\
Departamento de Matem\'atica,\\
~Universidade de Bras\'\i lia,\\
70910-900 Bras\'\i lia DF\\
Brazil

\end{document}